\documentclass[english]{article}

\usepackage{preamble}
\usepackage{mymacros}

\begin{document}

\title{Semiadditive Power Operations in Chromatic Height 1}
\author{Beckham Myers\thanks{Einstein Institute of Mathematics, Hebrew University of Jerusalem.}
}

\maketitle
\begin{abstract}
We recast the theory of power operations in chromatic height $1$ using higher semiadditivity. In particular, we show that if the height $1$ cyclotomic extension of a commutative algebra $R$ is Galois, then power operations on $R$-algebras are exactly $\delta$-ring operations. 

We also show that semiadditive power operations on chromatically localized algebraic $K$-theory agree with power operations defined using polynomial functoriality. As an application, we deduce that the map from the $K$-theory of a $\QQ$-linear symmetric monoidal category to its $K(1)$-localization is a $\delta$-ring homomorphism on the level of underlying spaces. It is well known that operations on the $K$-theory of $\QQ$-linear symmetric monoidal categories are $\lambda$-ring operations, and since $\delta$-rings are `$p$-typical' $\lambda$-rings, this is an instance of the redshift philosophy.

\end{abstract}
\tableofcontents

\section{Introduction}
\subsection{Overview}
Given a commutative ring spectrum $R$, its coherent multiplication gives rise to \emph{power operations} that act on the homotopy groups of commutative $R$-algebras. Power operations are an accessible and important part of the structure of commutative ring spectra.

The algebraic theory of power operations depends heavily on the ring spectrum $R$. When $R=\QQ$, the homotopy groups of $\QQ$-algebras are simply graded-commutative rings: no new operations appear. When $R=\FF_p$, power operations are governed by {the Dyer-Lashof algebra.}

In higher algebra, the \emph{chromatic filtration} organizes the category of spectra by \emph{chromatic height}, providing an infinite collection of characteristics that interpolate between classical characteristic $0$ and classical characteristic $p$. It is thus natural to consider how the theory of power operations interacts with the chromatic filtration.

Power operations in chromatic height $1$ are entirely understood due to work of McClure and Bousfield (\cite{hoo,bousfield}). Less is known about power operations in chromatic heights $n>1$, and the fundamental tool available is an algebro-geometric interpretation of power operations over Lubin--Tate spectra due to Ando, Hopkins, and Strickland (\cite{ando95,strickland97,strickland98,sigma}). The algebraic theory of power operations over Lubin--Tate spectra was further studied in \cite{rezk_congruence,completed_operations,null}. 

The first goal of this work is to begin investigating power operations in chromatic homotopy theory using \emph{higher semiadditivity}.
\subsubsection{Higher Semiadditivity}
 It is a remarkable result of Hopkins--Lurie (in the $K(n)$-local case) and Carmeli--Schlank--Yanovski (in the $T(n)$-local case) that the chromatic localizations are \emph{$\infty$-semiadditive} (\cite{HL,teleambi}). This is a vast generalization of the isomorphism between the coinvariants and invariants of a finite group action in characteristic $0$: an $\infty$-category is {$\infty$-semiadditive} when there is a canonical isomorphism
\[\Nm\colon\varinjlim_AF\iso \varprojlim_A F\]
between colimits and limits indexed by $\pi$-finite spaces.

Many chromatic phenomena can be placed in the framework of higher semiadditivity. For example, \emph{semiadditive height} is an invariant of higher semiadditive categories that reproduces chromatic height when applied to the chromatic localizations (\cite{ambiheight}). Whereas the chromatic height of a complex-oriented ring spectrum is defined via the associated formal group, semiadditive height is an invariant that depends only on the canonical higher semiadditive structure of an $\infty$-category. Because they are categorical in nature, semiadditive constructions are flexible and general, and enjoy good formal properties. As a consequence, higher semiadditivity has become an important tool in chromatic homotopy theory, particularly in the study of the $T(n)$-local categories, which are less computationally accessible than the $K(n)$-local categories (e.g. \cite{cyclo,fourier,telescope}).
\pagebreak

Given the role of higher semiadditivity in studying the $T(n)$-local categories, it is natural to wonder if $\Sp_{T(n)}$ can be characterized as the \textit{universal} stable (presentable) $\infty$-category of semiadditive height $n$, denoted by $\tsadi_n$.\footnote{This question was posed in a previous version of \cite{ambiheight}.} A counterexample of Yuan shows that this is not the case:
\begin{theorem}[\cite{yuan}]
    There exists a commutative algebra $\SS_p^{\tsadi_1}\in\CAlg(\tsadi_1)$ whose underlying spectrum is not $T(1)$-local.
\end{theorem}
Yuan also observes that in this case, the \emph{higher $p$th cyclotomic extension} $\SS_p^{\tsadi_1}[\omega_p^{(1)}]$ is not a Galois extension of $\SS_p^{\tsadi_1}$. This is an important property of $T(n)$-local commutative algebras---that their higher cyclotomic extensions are Galois---and it was studied in \cite{fourier} under the guise of \emph{virtual $\FF_p$-orientability}. This leads us to define the universal stable and presentable $ \infty$-category of height $n$ that is virtually $\FF_p$-orientable, which we denote by $\widehat\tsadi_n$ (\Cref{tsadibar}).

The following conjecture was posed by Tomer Schlank:
\begin{con}[T. Schlank]\label{conj}
    We have $\Sp_{T(n)}=\widehat\tsadi_n$. Namely, $\Sp_{T(n)}$ is characterized as the universal stable and presentable $\infty$-category that is virtually $\FF_p$-orientable of height $n$.
\end{con}
A positive answer to \Cref{conj} would give a characterization of the $T(n)$-local categories purely in terms of higher semiadditive structure, without any reference to the theory of formal groups.\footnote{We recently learned that Robert Burklund has unpublished work on \Cref{conj} in height $1$.}
\subsubsection{Semiadditive Power Operations in Height 1}
We now turn to the study of power operations and recall how higher semiadditivity can be used to produce interesting and nontrivial \emph{semiadditive power operations}. Given a finite group $G$ acting on $[n]$, there is a semiadditive power operation
\[x\longmapsto \int_{BG}x^n\]
 defined on the underlying spaces of commutative algebras in higher semiadditive $\infty$-categories.
Since they are defined purely formally, such operations exist very generally. Semiadditive power operations play a key role in the proof of the main result of \cite{teleambi}, and they were computed in certain cases in \cite{gwitt,keidar2025semiadditivealternatingpowerstwisted}.

There is a power operation $\theta$ on $E_n$-algebras that has a long history going back to work of Atiyah (\cite{atiyah}). Constructions of $\theta$ in all chromatic heights appear in work of Rezk and in \cite{mathew,staple}. In fact, the operation $\theta$ can be constructed as a semiadditive power operation:
\begin{proposition}[{\cite{teleambi,gwitt}, \Cref{theta_construction}}]
\label{theta_construction}
    Let $\cC$ be a $p$-complete and $1$-semiadditive symmetric monoidal stable $\infty$-category. For commutative algebras $R\in\CAlg(\cC)$, there exists a natural semiadditive power operation \[\theta\colon \underline R\too\underline R\] such that
    $\theta\colon\pi_0R\to\pi_0R$
    satisfies the equations
    \begin{enumerate}
    
       \item $\theta(1)=0$.
        \item
        $\theta(x+y)=\theta(x)+\theta(y)+\frac{x^p+y^p-(x+y)^p}{p}$ for all $x,y\in\pi_0R$.
    \end{enumerate}
\end{proposition}
\pagebreak
In chromatic height $1$, the operation $\theta$ satisfies the additional equation
\begin{enumerate}
       \setcounter{enumi}{2}
      \item $\theta(xy)=\theta(x)y^p + x^p\theta(y)+p\theta(x)\theta(y)$ for all $x,y\in\pi_0R$.\label{eq3}
    \end{enumerate}
A commutative ring equipped with an operation satisfying these three equations is called a \textit{$\delta$-ring}. In particular, for a $T(1)$-local commutative algebra $R$, $\pi_0R$ is a $\delta$-ring. We generalize this beyond the chromatic setting by relating the properties of $\theta$ to the higher semiadditive structure of $\Sp_{T(1)}$:
\begin{thmx}[{\Cref{theoremAproof}}]
\label{TheoremA}
    Let $\cC\in\CAlg(\Pr_\st)$ be $1$-semiadditive, and suppose $R\in\CAlg(\cC)$ has semiadditive height $1$. Further suppose that the cyclotomic extension $R[\omega_p^{(1)}]$ is a Galois extension of $R$. Then the operation $\theta$ of \Cref{theta_construction} is a $\delta$-ring operation. 
\end{thmx}

In general, the free commutative $R$-algebra $\Free_R^{\mathbb E_\infty}(t)$ on a degree-zero class $t$ classifies {power operations} that act on commutative $R$-algebras. Under the assumptions of \Cref{TheoremA}, we use the operation $\theta$ to give a complete description of such power operations. 
\begin{thmx}[{\Cref{theoremBproof}}]
\label{TheoremB}
     Let $\cC\in\CAlg(\Pr_\st)$ be $1$-semiadditive, and suppose $R\in\CAlg(\cC)$ has semiadditive height $1$. Further suppose that the cyclotomic extension $R[\omega_p^{(1)}]$ is a Galois extension of $R$. Then
    \[\pi_*\Free_R^{\mathbb E_\infty}(t)=(\pi_* R)[t,\theta(t),\theta^2(t),\dots],\]
    where $\theta$ is the operation of {\Cref{theta_construction}.}
\end{thmx}

In other words, $\pi_*\Free^{\mathbb E_\infty}_{R}(t)$ is the free $\delta$-ring over $\pi_*R$. This reproves and extends the description of power operations over $KU_p^\wedge$ due to McClure (\cite{hoo}) and stated in \cite{hopkins} for general $T(1)$-local commutative algebras.\footnote{We are not aware of a proof in the literature in the $T(1)$-local case, although the result is well known.} In place of arguments particular to $KU_p^\wedge$, our proof uses only higher semiadditive structure (and thus bypasses the descent from $KU_p^\wedge$ to $\SS_{T(1)}$).
\begin{rem}
    \Cref{TheoremB} is evidence for \Cref{conj}, as it shows that both $\Sp_{T(1)}$ and $\widehat\tsadi_1$ have the same theory of power operations.
\end{rem}

\subsubsection{Redshift and Semiadditive Power Operations}
Higher semiadditivity is also closely intertwined with the study of \emph{chromatically localized algebraic $K$-theory} (e.g. \cite{bms,descent}). There is a deep interaction between algebraic $K$-theory and the chromatic filtration on spectra, and a guiding principle in this interaction has been the \textit{redshift philosophy}, which roughly says that algebraic $K$-theory increases chromatic height by $1$ (\cite{AR02,AR08}). In recent years, several breakthroughs have collectively verified that $K$-theory indeed raises chromatic height by exactly $1$ for all commutative algebras (\cite{Land_2024,clausen2022descentvanishingchromaticalgebraic,null}).

Since algebraic $K$-theory raises chromatic height, one can approach chromatic height $n+1$ via the algebraic $K$-theory of height $n$ commutative ring spectra (e.g. \cite{cardinalities,telescope,benmoshe2025chromatichighersemiadditivityheight}). Algebraic $K$-theory is built via \emph{categorification}: the passage from a ring spectrum $R$ to its category of (left) modules $\LMod_R$. More conceptually, this suggests that chromatic height $n+1$ is related to the categorification of chromatic height $n$ (e.g. \cite{gwitt, keidar2025semiadditivealternatingpowerstwisted}).
\pagebreak

We next study semiadditive power operations from the perspective of redshift. Let $R$ be a $T(n)$-local commutative ring spectrum. For a finite group $G$ acting on $[n]$, the functor
\[(-)^{\otimes n}_{hG}\colon \Mod_R\too\Mod_R\]
extends to an operation on the underlying space of $K(R)$ using the polynomial functoriality of algebraic $K$-theory (\cite{polynomial}).

We show that this operation, arising from the categorification $\Mod_R$, agrees with an operation arising from the higher semiadditive structure of $\Sp_{T(n+1)}$.
\begin{thmx}[{\Cref{ktheory_comparison}, \Cref{lnf_comparison}}]
\label{lnf_comparison_intro}
    Suppose $R$ is a $T(n)$-local commutative ring spectrum. For a finite $p$-group $G$ acting on $[n]$, there is a commutative square of spaces 
\[\begin{tikzcd}[ampersand replacement=\&,column sep=large]
	{\Omega^\infty K(R)} \& {\Omega^\infty K(R)} \\
	{\Omega^\infty K_{T(n+1)}(R)} \& {\Omega^\infty K_{T(n+1)}(R).}
	\arrow["{(-)_{hG}^{\otimes n}}", from=1-1, to=1-2]
	\arrow[from=1-1, to=2-1]
	\arrow[from=1-2, to=2-2]
	\arrow["{\int_{BG}(-)^n}", from=2-1, to=2-2]
\end{tikzcd}\]
\end{thmx}
The main difficulty in proving \Cref{lnf_comparison_intro} lies in the fact that the operations being compared are defined quite differently. We accomplish this by constructing the lax symmetric monoidal structure on $K$-theory using polynomial functoriality, a construction that may be of independent interest (\Cref{lax_sm}).
    \begin{rem}Suppose $\cC\in\CAlg(\Cat_\perf)$ is $L_n^f$-local and has $1$-finite colimits. For $G$ acting on $[n]$ and $V\in\cC^{BG}$, there is an operation on $K(\cC)$ induced by $(V\otimes(-)^{\otimes n})_{hG}$. More generally, we show that this operation is a lift of the power operation 
   \[x\longmapsto \int_{BG}[V]\cdot x^n\]
   defined on $K_{T(n+1)}(\cC)$,
   where $[V]$ is the class of $V$ in $K_{T(n+1)}(\cC^{BG})\simeq K_{T(n+1)}(\cC)^{BG}$.
\end{rem}
These results give a higher semiadditive construction of $T(n+1)$-local power operations defined using polynomial functoriality.
\subsubsection{Redshift and $\lambda$-Rings - The Case of Height $1$}
We now consider the relationship between operations on $\Mod_R$ and $T(n+1)$-local power operations when $n=0$. Suppose $R$ is a rational commutative ring spectrum. Classically, operations on $\Mod_R$ are given by the theory of \textit{Schur functors}. The action of Schur functors on $\Mod_R$ extends to $K$-theory, giving $K_0(R)$ the structure of a \textit{$\lambda$-ring}\footnote{For an overview of $\lambda$-rings see \cite{knutson,haze,donald}.} (\Cref{rational_lambda}). In fact, the definition of a $\lambda$-ring axiomatizes this action of Schur functors (\cite{schur}). 

By work of Joyal and Rezk, a $\lambda$-ring is precisely a commutative ring equipped with compatible $\delta$-ring structures $\delta_p$ for every prime $p$ (\cite{joyal_lambda,rezk_lambda}, and see \cite{hubner}). This means that a $\delta$-ring (at a fixed prime $p$) should be considered the `$p$-typical' version of a $\lambda$-ring. Since power operations in height $1$ are $\delta$-ring operations, this is an instance of the redshift philosophy: power operations in height $1$ are exactly the $p$-typical version of power operations on the $K$-theory of rational commutative ring spectra. 
\pagebreak

We conclude by showing that the localization map on algebraic $K$-theory indeed intertwines these $p$-typical $\lambda$-ring structures:
\begin{corx}[\Cref{redshift_delta}]
\label{TheoremC}
    Let $\cC\in\CAlg(\Cat_\perf)$ be $\QQ$-linear. 
     Then there is a commutative square of spaces
\[\begin{tikzcd}[ampersand replacement=\&]
	{\Omega^\infty K(\cC)} \& {\Omega^\infty K(\cC)} \\
	{\Omega^\infty K_{T(1)}(\cC)} \& {\Omega^\infty K_{T(1)}(\cC),}
	\arrow["{\delta_p}", from=1-1, to=1-2]
	\arrow[from=1-1, to=2-1]
	\arrow[from=1-2, to=2-2]
	\arrow["\theta", from=2-1, to=2-2]
\end{tikzcd}\]
where $\delta_p$ is the $p$-derivation associated to the $\lambda$-ring structure on $K(\cC)$ (\Cref{rational_lambda}) and $\theta$ is the operation of \Cref{theta_construction}.
\end{corx}

\subsection{Conventions}
\begin{itemize}

\item A \textit{category} is an $(\infty,1)$-category and a \textit{$1$-category} is a $(1,1)$-category.
\item $\Pr$ is the category of presentable categories and left adjoints, often denoted by $\Pr^L$.
\item For $0\le m \le\infty$, we let $\Cat^{m\text{-fin}}:=\Cat^{\mathcal K_m}$, where $\mathcal K_m$ is the collection of $m$-finite spaces (i.e. $m$-truncated spaces with finitely many connected components, each of which with finite homotopy groups). 
\item For $0\le m\le \infty$, $\Cat^{m\text{-}\oplus}\subset \Cat^{m\text{-fin}}$ is the full subcategory on $m$-semiadditive categories.
\item $\Cat_\perf$ is the category of idempotent-complete (small) stable categories.
    \item It is standard to use the symbol `$\delta$', and not `$\theta$', to denote a $p$-derivation. We use `$\delta$' when discussing a {general} $p$-derivation and denote the operation of \Cref{theta_construction} by `$\theta$.'
\end{itemize}
\subsection{Acknowledgements}
I thank my advisor Lior Yanovski for suggesting many ideas reflected in this work, particularly the statement of \Cref{vo_height0}, and for helpful comments on earlier drafts of this paper. I also thank Tomer Schlank, Dorin Boger, and the entire Jerusalem Seminarak group for useful discussions about semiadditive power operations in general. Finally, I thank Omer Bojan for many helpful and interesting conversations.

AI tools were used to slightly simplify the proof of \Cref{injective}. All other mathematical content (and all writing) is my own.
\section{Higher Semiadditivity}
\label{higher_semiadditivity}
In this section we briefly recall the theory of higher semiadditivity (see \cite{HL,teleambi}). We then introduce the higher cyclotomic extensions and Fourier theory of \cite{cyclo,fourier}. We discuss virtual orientability and introduce the universal virtually $\FF_p$-orientable stable presentable category of height $n$. We conclude the section by defining semiadditive power operations.
\subsection{Higher Commutative Monoids}
\label{higher_comm_monoids}

A {commutative monoid} in a category $\cC$ is an object $M\in \cC$ equipped with a way to sum finitely many `points' in $M$. The data of a commutative monoid can be succinctly packaged as a functor
\[\Span(\Fin)\too\cC\]
that satisfies the \textit{Segal condition}, where $\Span(\Fin)$ is the $(2,1)$-category of spans of finite sets.     Here, the Segal condition is equivalent to the condition that the functor $\Span(\Fin)\to\cC$ preserves finite products.

More generally, a \textit{higher} commutative monoid in $\cC$ is an object $M\in\cC$ equipped with a way to sum families of points indexed by \textit{spaces}.
\begin{defn}[{\cite{spans}}]
   An \textit{\mdef{$m$-commutative monoid}} in a category $\cC$ is a functor
    \[M\colon\Span(\Spc_{m\textup{-fin}})\too\cC\]
    that preserves $m$-finite limits. The category of $m$-commutative monoids in $\cC$ is given by
    \[\mdef{\CMon_m(\cC):=\Fun^{m\text{-fin}}(\Span(\Spc_{m\textup{-fin}}),\cC)},\]
    where $\Fun^{m\textup{-fin}}(-,-)$ denotes the full subcategory on functors preserving $m$-finite limits.
    \end{defn}

Recall that a semiadditive category $\cC$ is a pointed category in which products and coproducts agree. In a semiadditive category, every object is canonically a commutative monoid. Roughly, a category is \textit{$m$-semiadditive} when limits and colimits over $m$-finite \textit{spaces} agree (\cite{HL}). If this holds, then every object of $\cC$ is canonically a higher commutative monoid:
\begin{theorem}[{\cite{spans}}]
    Suppose $\cC$ is an $m$-semiadditive category. Then the functor
    \[\CMon_m(\cC)\iso  \cC\]
    given by evaluation at $\pt\in \Span(\Spc_{m\textup{-fin}})$
    is an equivalence.
\end{theorem}
This is very useful, as it means that a \textit{property} of a category (being higher semiadditive) implies the existence of \textit{structure} on its objects. Therefore, we think about $m$-semiadditivity of $\cC$ as providing a certain {interface}: every object of $\cC$ is canonically a higher commutative monoid. 

We can unwrap the structure of a higher commutative monoid a bit. Given an $m$-commutative monoid $M$, its underlying object is $M(\pt)$. For a map $q\colon A\to B$ of $m$-finite spaces, the span $B\leftarrow A\rightarrow A$ gives rise to restriction maps
\[q^*\colon M^B\too M^A,\]
and the span $A\leftarrow A\to B$ gives rise to transfers
\[q_*\colon M^A\too M^B.\]
Informally, the pushforward $q_*$ encodes summation, or \textit{integration}, along the fibers of $q$. We denote it by the `semiadditive integral' 
\[\mdef{\int_q(-):=q_*(-)}.\]
This notation is meant to suggest various identities and compatibilities. Indeed, the semiadditive integral satisfies homogeneity, Fubini's theorem, etc. (\cite[Section 2]{teleambi}).

As an example of the formal calculus of the semiadditive integral, we show that $q_*$ and $q^*$ are dual with respect to the `integration pairing':
\begin{proposition}
    \label{dirac}
    Let $\cC$ be an $m$-semiadditive symmetric monoidal category,  
    and let $q\colon A\to B$ be a map of $m$-finite spaces. For every $R\in\CAlg(\cC)$ and pair of functions $f\in \pi_0R^A$ and $g\in \pi_0R^B$,
    \[\int_Bq_*(f)\cdot g=\int_Af\cdot q^*(g).\]
    
\end{proposition}
\begin{proof}
Since
\[\int_Af\cdot q^*(g)=\int_B\int_qf\cdot q^*(g)\]
by Higher Fubini's Theorem (\cite[Proposition 2.1.15]{teleambi}), we are reduced to showing
\[q_*(f)\cdot g=\int_qf\cdot q^*(g).\]
    This equality follows from homogeneity of the semiadditive integral (\cite[Proposition 2.1.14]{teleambi}) applied to the $R$-linear homomorphism $R^B\to R^B$ given pointwise by $g(b)$.
\end{proof}
\subsubsection{Semiadditive Height}
We recall the following notions from \cite{ambiheight}.
\begin{defn}
    Let $M$ be an $m$-commutative monoid. For every $m$-finite space $A$, the \mdef{\textit{cardinality}} of $A$ is the natural endomorphism
    \[\mdef{|A|\colon M(\pt)\too M(\pt)}\]
    defined by the span $\pt\leftarrow A\to \pt$.
\end{defn}
When $A$ is a finite set, $|A|$ is indeed multiplication by the cardinality of $A$. Cardinalities are used to define invariants of objects in a higher semiadditive category: 
\begin{defn}
    For a category $\cC$ and natural endomorphism $\alpha\colon\id_\cC\to\id_\cC$, we say that
    \begin{enumerate}
        \item $X\in\cC$ is \mdef{\emph{$\alpha$-divisible}} if $\alpha_X$ is an isomorphism.
        \item $X\in\cC$ is \mdef{\emph{$\alpha$-complete}} if $\Map(Z,X)=0$ for all $\alpha$-divisible $Z\in\cC$.
    \end{enumerate}
\end{defn}
\begin{defn}
    Let $\cC$ be an $m$-semiadditive category. Then for every $n\le m$, we say that an object $X\in\cC$ has \mdef{\textit{semiadditive height $n$}} (at an implicit prime $p$) when $X$ is $|B^{n-1}C_p|$-complete and $|B^nC_p|$-divisible.
    
    An $m$-semiadditive category $\cC$ has height $n\le m$ when all of its objects have height $n$.
\end{defn}
\begin{example}[{\cite[Theorem 4.4.5]{ambiheight}}] The categories $\Sp_{K(n)}$ and $\Sp_{T(n)}$ have semiadditive height $n$.

\end{example}
\subsection{Higher Cyclotomic Extensions and Virtual Orientability}
\label{cyclo_virt}

Semiadditive height, as a generalization of chromatic height, is an invariant which detects the `characteristic' of objects in a higher semiadditive category. We now introduce an additional condition defined using semiadditive structure.

If $R$ is a discrete commutative ring, then one can formulate the following condition:
\begin{quote}
    $(\ast)$ the $p$th cyclotomic extension $R[\zeta_p]$ of $R$ is a Galois extension of $R$.
\end{quote}
Condition $(\ast)$ is \textit{equivalent} to the condition that $p$ acts invertibly on $R$, or equivalently, that $R$ has semiadditive height $0$.

In \cite{cyclo}, the authors define \textit{higher} cyclotomic extensions $R[\omega_p^{(n)}]$ of height $n$. When $n>0$, the extension $R[\omega_p^{(n)}]$ may fail to be Galois, even when $R$ has semiadditive height $n$ (\cite[Proposition~3.9]{yuan}). In this section we recall how the condition that $R[\omega_p^{(n)}]$ is Galois can be re-expressed using the higher Fourier theory of \cite{fourier}, which then naturally leads to the notion of \emph{virtual orientability}.

\subsubsection{Cyclotomic Extensions}
\label{cyclo_virtual}
\begin{defn}[{\cite{cyclo}}]
Let $\cC$ be a symmetric monoidal category. A \mdef{\textit{(height $n$) $p$th root of unity}} of $R\in\CAlg(\cC)$ is a map of connective spectra \[\omega\colon \Sigma^n\FF_p\to R^\times.\] 

We say $\omega$ is \mdef{\textit{primitive}} if the only commutative $R$-algebra $S\in \CAlg(\cC)$ for which the diagram
\[\begin{tikzcd}[ampersand replacement=\&]
	{\Sigma^n\FF_p} \& {R^\times} \\
	0 \& {S^\times}
	\arrow[from=1-1, to=1-2]
	\arrow[from=1-1, to=2-1]
	\arrow[from=1-2, to=2-2]
	\arrow[from=2-1, to=2-2]
\end{tikzcd}\]
commutes is $S=0$.
\end{defn}
When $\cC$ is presentable, stable, and $R$ has semiadditive height $n$, the \textit{higher cyclotomic extension} \mdef{$R[\omega_p^{(n)}]$} is an $\FF_p^\times$-extension of $R$ that is equipped with a universal primitive (height $n$) $p$th root of unity (\cite{cyclo}).

Since the construction of cyclotomic extensions requires $\cC$ to be stable, there is no way to make sense of them for commutative algebras like $\Mod_R\in\CAlg(\Pr)$, which are not objects in a stable category. For this reason, we instead work with another, more flexible notion coming from the higher Fourier theory of \cite{fourier}.

 \subsubsection{Higher Fourier Theory}
We briefly recall the relevant material from \cite{fourier}. 

The space of (height $n$) $p$th roots of unity $
\omega$ is naturally isomorphic to the space of \textit{semiadditive Fourier transforms} $\mathfrak F_\omega$, which are transformations 
\[\mdef{\mathfrak F_\omega\colon R[M]\too R^{\lfloor I^{(n)}M\rfloor}}\qin\CAlg(\cC)\]
natural in $\pi$-finite, $n$-truncated $\FF_p$-modules $M\in\Mod_{\FF_p}^{[0,n]\text{-fin}}(\Sp_{\textup{cn}})$. Here, the spectrum $I^{(n)}$ is a shifted spectral version of Pontryagin duality.

Even when $\omega$ is a primitive root of unity, the Fourier transform $\mathfrak F_\omega$ need not be an isomorphism. In height $1$, this is related to the condition that the cyclotomic extension $R[\omega_p^{(1)}]$ is Galois:
\begin{proposition}[{\cite[Proposition 6.13]{fourier}}]
\label{cyclo_galois}
    Let $\cC\in\CAlg(\Pr^{1\textup{-}\oplus}_\st)$ and $R\in\CAlg(\cC)$, and let $\omega\colon\Sigma\FF_p\to R[\omega_p^{(1)}]^\times$ be the tautological primitive (height $1$) $p$th root of unity. Then the extension $R[\omega_p^{(1)}]$ is a $\FF_p^\times$-Galois extension of $R$ if and only if the Fourier transform $\mathfrak F_\omega$ is an isomorphism.
\end{proposition}
\begin{rem}
\label{general_height_galois}
    In height $n$, the extension $R[\omega_p^{(n)}]$ is Galois if and only if the Fourier transform associated to the tautological (height $n$) $p$th root of unity of $R[\omega_p^{(n)}]$ is an isomorphism when evaluated on \textit{discrete} $\FF_p$-modules $M\in\Mod_{\FF_p}^{[0,n]\text{-fin}}(\Sp_\textup{cn})$.
\end{rem}

It is natural to formulate the condition that the Fourier transform $\mathfrak F_\omega$ is an isomorphism on \textit{all} truncated $\FF_p$-modules, not just discrete ones.
\begin{defn}
    A (height $n$) $p$th root of unity $\omega$ is an \mdef{\textit{orientation}}\footnote{We use the term `orientation' (and similarly `orientable') to refer to $\FF_p$-orientations, since we will not consider orientability with respect to ring spectra other than $\FF_p$.} if the corresponding Fourier transform $\mathfrak F_\omega$ is an isomorphism when evaluated on {all} $\FF_p$-modules $M\in\Mod_{\FF_p}^{[0,n]\text{-fin}}(\Sp_\textup{cn})$. If $R\in\CAlg(\cC)$ admits an orientation, then $R$ is called \textit{\mdef{orientable} (of height $n$).}
\end{defn}
\begin{example}[{\cite[Example 4.2]{fourier}}]
\label{height0_roots}
    A $p$th root of unity of height $0$ is an orientation if and only if it is primitive.
\end{example}
Often, $R$ admits an orientation only after faithful extension of scalars. In this case, many pleasant consequences of orientability still hold by \textit{faithful descent} (\cite[Section~4.3]{fourier}).
\begin{defn}
    We say $R\in\CAlg(\cC)$ is \textit{\mdef{virtually orientable} (of height $n$)} if there exists a faithful commutative $R$-algebra $S\in\CAlg(\cC)$ that is orientable (of height $n$).
    The category $\cC$ is virtually orientable if $\one_\cC$ is.\end{defn}

   \subsubsection{Examples}
  We thank Lior Yanovski for the following example.
   \begin{proposition}
   \label{vo_height0}
       Suppose $\cC\in\CAlg(\Pr^{0\textup{-}\oplus})$ has semiadditive height $0$. Then $\cC$ is virtually orientable if and only if it is additive.
   \end{proposition}
   \begin{proof}Suppose $\cC$ is virtually orientable. $\cC$ is additive precisely when the shear map
        \[\begin{bmatrix}
            1&1\\
            0&1
        \end{bmatrix}\colon X\oplus X\too X\oplus X\]
        is an isomorphism on all objects. It is enough to check this on the unit $\one_\cC$, and furthermore it suffices to check this after faithful extension of scalars. So we assume $\cC$ is orientable and let $\omega\colon \FF_p\to\one_\cC^\times$ be an orientation of height $0$. Equivalently, $\omega$ is a primitive root of unity. Thus there exists an endomorphism
        \[-1:=\omega+\omega^2+\ldots+\omega^{p-1}\]
        of $\one_\cC$ such that $1+(-1)=0$. The inverse to the shear map is then given by
        \[\begin{bmatrix}1&-1\\0&1\end{bmatrix}:\one_\cC\oplus\one_\cC\too\one_\cC\oplus\one_\cC.\]
   Conversely, suppose now that $
   \cC$ is additive. Then there is an idempotent
         \[\varepsilon=1-\frac{1}{p}\big(1+\sigma+\ldots+\sigma^{p-1}\big)\]
         in $\one_\cC[C_p]$ such that $\one_\cC[C_p][\varepsilon^{-1}]=:\one_\cC[\omega_p^{(0)}]$ is the $p$th cyclotomic extension. Since the cyclotomic extension is faithful and admits a primitive root of unity, the claim follows from \Cref{height0_roots}.
    \end{proof} 
    We next record the fact that virtual orientability interacts well with categorification. We first state a preliminary lemma.
  \begin{lemma}
  \label{ambi_faithful}
      Let $R\in\CAlg(\cC)^{BG}$, and suppose $BG$ is $\cC$-ambidextrous. Then $R^{hG}\to R$ is faithful.
  \end{lemma}
  \begin{proof}
  Let $M$ be a $R^{hG}$-module such that $R\otimes M\simeq0$. Then
  \[0\simeq (R\otimes M)^{hG}\simeq R^{hG}\otimes M\]
  by the projection formula (\cite[Corollary 2.34]{fourier}), as desired.
  \end{proof}
  
  \begin{proposition}
  \label{virtual_categorification}
      Let $\cC\in\CAlg(\Pr^{n\textup{-}\oplus})$ and suppose that $R\in\CAlg(\cC)$ is virtually orientable of height $n$. Then $\Mod_R(\cC)\in\CAlg_\cC(\Pr)$ is virtually orientable of height $n+1$.
  \end{proposition}
  \begin{proof}
      Let $S\in\CAlg(\cC)$ be the universal orientable commutative $R$-algebra in $\cC$, which is a faithful $\FF_p^\times$-Galois extension of $R$ (\cite[Proposition 4.22, Proposition 4.24]{fourier}). Then $\Mod_S(\cC)$ is orientable of height $n+1$ (\cite[Corollary 5.16]{fourier}). As $R\to S$ is a faithful Galois extension, Galois descent provides an equivalence
      \[\Mod_R(\cC)\iso \Mod_S(\cC)^{h\FF_p^\times}.\]
      It follows that $\Mod_R(\cC)\to\Mod_S(\cC)$ is a faithful map of commutative algebras in $\Mod_\cC(\Pr)$ by \Cref{ambi_faithful}.
  \end{proof}
  Another family of examples consists of the chromatic localizations.
  \begin{proposition}[{\cite[Proposition 7.27]{fourier}}]
\label{tn_vo}
      The category $\Sp_{T(n)}$ is virtually orientable of height~$n$.
  \end{proposition}
  \pagebreak
  Together, the previous two propositions imply the following:
  \begin{corollary}
   \label{modr_virtual}
      Let $R$ be a $T(n)$-local commutative ring spectrum. Then the commutative algebra $\Mod_R(\Sp_{T(n)})\in\CAlg(\Pr)$ is virtually orientable of height $n+1$.
  \end{corollary}
  \begin{proof}
  Since $R$ is $T(n)$-local, it is virtually orientable of height $n$ (\Cref{tn_vo}). By \Cref{virtual_categorification}, $\Mod_R\in\CAlg(\Pr_{T(n)})$ is virtually orientable of height $n+1$. 
  
  The forgetful functor $\Pr_{T(n)}\to\Pr$ preserves virtual orientability, as it preserves orientability by \cite[Proposition 4.25]{descent} and also preserves faithful morphisms. Thus $\Mod_R$ is virtually orientable in $\CAlg(\Pr)$ as well.
  \end{proof}

    \subsubsection{Universal Virtually Orientable Categories}
   In this subsection, we construct the universal example of a virtually orientable category of height $n$. Recall that there is an idempotent algebra $\tsadi_n\in\CAlg(\Pr)$ that classifies the property of being stable, $p$-local, and $\infty$-semiadditive of height $n$ (\cite[Theorem 5.3.6]{ambiheight}). 
    
    Denote $\one:=\one_{\tsadi_n}$, and recall that $\one[\omega_{\FF_p}^{(n)}]$ is the universal commutative algebra in $\tsadi_n$ equipped with an orientation of height $n$ (constructed in \cite[Section~4.2]{fourier}).\footnote{Note that $\one[\omega_{\FF_p}^{(n)}]$ is \emph{not} the higher cyclotomic extension $\one[\omega_p^{(n)}]$.}

    \begin{defn}Let $\mdef{\widehat\tsadi_n}$ be the Bousfield localization of $\tsadi_n$ with respect to $\one[\omega_{\FF_p}^{(n)}]$.
    \end{defn}

    We first show that $\tsadi_n\to\widehat\tsadi_n$ is a particularly simple type of localization (compare to \cite[Proposition 6.19]{fourier}).
\begin{proposition}
    The localization $\tsadi_n\to\widehat\tsadi_n$ is smashing, with unit $\one[\omega_{\FF_p}^{(n)}]^{h\FF_p^\times}$.
\end{proposition}
\begin{proof}
      As in \cite[Section 4.2]{fourier}, the commutative algebra $\one[\omega_{\FF_p}^{(n)}]$ has an $\FF_p^\times$-action. We will prove that the localization $\tsadi_n\to\widehat\tsadi_n$ is given by the formula
      \[X\longmapsto (X\otimes\one[\omega_{\FF_p}^{(n)}])^{h\FF_p^\times}\simeq X\otimes\one[\omega_{\FF_p}^{(n)}]^{h\FF_p^\times}\]
      by showing that $\one[\omega_{\FF_p}^{(n)}]^{h\FF_p^\times}$ is an idempotent algebra in the same Bousfield class as $\one[\omega_{\FF_p}^{(n)}]$.

     We first show they lie in the same Bousfield class. This is because $\one[\omega_{\FF_p}^{(n)}]$ is faithful over $\one[\omega_{\FF_p}^{(n)}]^{h\FF_p^\times}$ by \Cref{ambi_faithful}, so that $X$ is $\one[\omega_{\FF_p}^{(n)}]^{h\FF_p^\times}$-acyclic if and only if it is $\one[\omega_{\FF_p}^{(n)}]$-acyclic.

      We now show that $\one[\omega_{\FF_p}^{(n)}]^{h\FF_p^\times}$ is idempotent. Since the algebra $\one[\omega_{\FF_p}^{(n)}]$ admits an orientation of height $n$, after base change along $\one\to\one[\omega_{\FF_p}^{(n)}]$ the unit $\one\to\one[\omega_{\FF_p}^{(n)}]^{h\FF_p^\times}$ becomes an isomorphism
      
            \[
            \one\otimes\one[\omega_{\FF_p}^{(n)}]\iso \one[\omega_{\FF_p}^{(n)}]^{h\FF_p^\times}\otimes\one[\omega_{\FF_p}^{(n)}].
            \]
      Taking $\FF_p^\times$-fixed points gives the result.
        \end{proof}
        By \cite[Proposition~4.27]{fourier}, $\widehat\tsadi_n$ is the initial symmetric monoidal {localization} of $\tsadi_n$ that is virtually orientable of height~$n$. We now show that $\widehat\tsadi_n$ is the universal such category among \emph{all} presentably symmetric monoidal categories (and not merely those which are localizations of $\tsadi_n$).
    \begin{proposition}
    \label{tsadibar}
    ${\widehat\tsadi_n}\in\CAlg(\Pr)$ is the universal stable, $p$-local, $\infty$-semiadditive, and virtually orientable category of height $n$.
    \end{proposition}
        \begin{proof}
               Let $\cC\in\CAlg(\Pr)$ be another such category. Since $\cC$ is a module over the idempotent algebra $\tsadi_n\in\CAlg(\Pr)$, there is a unique colimit-preserving symmetric monoidal functor
            \[F\colon\tsadi_n\too\cC.\]
           Since $\widehat\tsadi_n$ is a Bousfield localization of $\tsadi_n$, there is at most one colimit-preserving symmetric monoidal functor $\widehat\tsadi_n\to\cC$. To prove that there indeed exists such a functor, we will show that $F$ vanishes on all $\one[\omega_{\FF_p}^{(n)}]$-acyclics $X$. For this, we claim it is enough to show that $F(\one[\omega_{\FF_p}^{(n)}])\in\CAlg(\cC)$ is faithful. Indeed, this is because
           \[0\simeq F(X\otimes\one[\omega_{\FF_p}^{(n)}])\simeq F(X)\otimes F(\one[\omega_{\FF_p}^{(n)}]).\]
           To show $F(\one[\omega_{\FF_p}^{(n)}])$ is faithful, let $G$ be the right adjoint of $F$, and let $S$ be the universal orientable commutative algebra in $\cC$. As a lax symmetric monoidal $\infty$-semiadditive functor, $G$ takes the Fourier transform over $S$ to the Fourier transform over $G(S)$ (\cite[Proposition 4.25]{descent}). In particular, $G(S)$ is orientable, so that there is a map of commutative algebras
           \[\one[\omega_{\FF_p}^{(n)}]\too G(S)\qin\CAlg(\tsadi_n),\]
           which corresponds to a map of commutative algebras
           \[F(\one[\omega_{\FF_p}^{(n)}])\too S\qin\CAlg(\cC).\]
           Since $\cC$ is virtually orientable, $S$ is faithful (\cite[Proposition 4.24]{fourier}), and thus $F(\one[\omega_{\FF_p}^{(n)}])$ is faithful as well.
\end{proof}

        \subsection{Semiadditive Power Operations}
        \label{semiadditive_power_operations}
In this subsection, we introduce a class of power operations coming from the higher semiadditive structure of a symmetric monoidal category $\cC$.
    \label{semiadditive_powers}
   
    \begin{defn}
        Let $\cC$ be a symmetric monoidal category, and let $R\in\CAlg(\cC)$. The \textit{$n$th total power operation} is the map of spaces
        \[\mdef{\mathcal P_n\colon \underline R\too\underline R^{B\Sigma_n}}\]
        defined as the composition
        \[\underline R:=\Map(\one_{\cC},R)\xlongrightarrow{(-)^{\otimes n}}\Map(\one_{\cC},R^{\otimes n})^{h\Sigma_n}\iso \Map(\one_\cC,(R^{\otimes n})^{h\Sigma_n})\too\Map(\one_\cC,R^{B\Sigma_n})=:\underline R^{B\Sigma_n}.\]
    \end{defn}
    
    \begin{rem}
     We think of $\mathcal P_n$ as the $n$th \textit{equivariant} power
        \[x\longmapsto x^n,\]
        where we remember the $n$-fold symmetry of $x^n$. For this reason, we sometimes denote $\mathcal P_n(x)$ by $x^n$ when the meaning is clear. The function $\mathcal P_n$ is multiplicative, but not additive.
    \end{rem}
        
    When the category $\cC$ is $1$-semiadditive, a commutative algebra $R\in\CAlg(\cC)$ has {semiadditive power operations} obtained by composing $\mathcal P_n$ with the semiadditive integral:
\begin{defn}
\label{semiadditive_powers_defn}
Let $\cC$ be a $1$-semiadditive symmetric monoidal category, and let $R\in\CAlg(\cC)$. For a finite group $G$ acting on $[n]$ and $V\in \pi_0 R^{BG}$, we define the operation
    
    \begin{align*}\mdef{\beta_G^V:\underline R\>}&\mdef{\too\underline R}
       \end{align*}
       by
       \[\beta_G^V(x):=\int_{BG}Vx^n.\]
       Formally, $\beta_G^V$ is the composition \[\underline R\xlongrightarrow{\mathcal P_n}\underline R^{B\Sigma_n}\too \underline R^{BG}\xlongrightarrow{V\cdot(-)} \underline R^{BG}\xlongrightarrow{\int_{BG}(-)} \underline R.\]
\end{defn}
The operations $\beta_G^V$ are natural in maps of commutative algebras. Namely, if $f\colon R\to S$ is a map of commutative algebras, then there is a commutative square of spaces
\[\begin{tikzcd}[ampersand replacement=\&,column sep=large]
	{\underline R} \& {\underline R} \\
	{\underline S} \& {\underline S.}
	\arrow["{\beta_G^V}", from=1-1, to=1-2]
	\arrow[from=1-1, to=2-1]
	\arrow[from=1-2, to=2-2]
	\arrow["{\beta_G^{f(V)}}", from=2-1, to=2-2]
\end{tikzcd}\]

This definition encompasses semiadditive power operations that have previously been considered in the literature:
\begin{example}
\label{alpha}
In \cite{teleambi}, the authors consider the semiadditive power operation
    \[\alpha(x):=\int_{BC_p}x^p.\]
    Integration over $BC_p$ should be thought of as `dividing by the $C_p$-action', so that the operation $\alpha(x)$ is a version of the operation `$x\mapsto \frac{x^p}{p}$'. Indeed, they show that it satisfies the equation
\[\alpha(x+y)=\alpha(x)+\alpha(y)+\frac{(x+y)^p-x^p-y^p}{p}.\]

 \end{example}

   \begin{example}
       In \cite{keidar2025semiadditivealternatingpowerstwisted}, the authors define \textit{twisted power operations}. A twisted power operation is an operation $\beta_G^V$ for which $V$ is a (higher) character of the group $G$.
   \end{example}
\section{Algebraic $K$-Theory}
In this section we relate the semiadditive power operations of \Cref{semiadditive_power_operations} to power operations on algebraic $K$-theory defined using polynomial functoriality.

   \subsection{Polynomial Operations on Algebraic $K$-Theory}
\label{poly_operations_k_theory}
Recall that the algebraic $K$-theory space of a ring naturally extends to an invariant \[\mdef{\uK\colon \Cat_\perf\to\Spc}\] of stable (idempotent-complete) categories. The functor $\uK$ admits a unique lift to (connective) spectra, which is the algebraic $K$-theory functor \[\mdef{K\colon\Cat_\perf\to \Sp}.\] The algebraic $K$-theory of a ring spectrum $R$ is then recovered by setting
\[\mdef{K(R):=K(\LMod_R^\text{ldbl}).}\]

$\uK$ is functorial more generally in polynomial functors:
\begin{theorem}[{\cite{polynomial}}]
\label{polynomial_functoriality}
Let $\mdef{\Cat_\perf^\poly}$ be the category of idempotent-complete stable categories and polynomial functors between them (in the sense of \cite[Definition 2.11]{polynomial}). Then $\uK$ extends to a functor
    \[\uK\colon\Cat_\perf^\poly\too\Spc.\]
\end{theorem}
\begin{rem}[{\cite[Proposition 2.15]{polynomial}}]
\label{characterize_polynomial}
    A functor between stable idempotent-complete categories is polynomial of degree $\le n$ if and only if it is $n$-excisive (in the sense of Goodwillie calculus).
\end{rem}
We now explain how \Cref{polynomial_functoriality} can be used to extend (non-additive) operations to the $K$-theory space of stable categories. We begin by recalling that the $n$th tensor power is a polynomial functor of degree $n$, as expected.
\begin{lemma}
\label{nth_power_polynomial}
    Let $\cC\in\Cat_\perf$ be symmetric monoidal. The functor $(-)^{\otimes n}:\cC\to\cC^{B\Sigma_n}$ is polynomial of degree $n$.
\end{lemma}
\begin{proof}
    The functor $(-)^{\otimes n}$ is given by the composition
    \[\cC\too (\cC^n)^{h\Sigma_n}\too(\cC^{\otimes n})^{h\Sigma_n}\too
    \cC^{B\Sigma_n}.\]
    The first and last functors are exact and thus linear. The middle functor is given by taking the fixed points of the $\Sigma_n$-action on the functor
    \[\textup{can}\colon \cC\times\ldots\times\cC\too\cC\otimes\ldots\otimes\cC.\]
    Since $\textup{can}$ is exact in each variable, it is polynomial of degree $n$ when viewed as a functor $\cC^n\to\cC^{\otimes n}$ (\cite[6.1.3]{HA}). Since polynomial functors of degree $n$ are closed under taking fixed points, it follows that the above composition is also polynomial of degree $n$.
\end{proof}
\begin{defn}
\label{beta_hat}
    Let $\cC\in\Cat_\perf$ be symmetric monoidal with $1$-finite colimits. For a finite group $G$ acting on $[n]$ and $V\in\cC^{BG}$, we define the functor
    \begin{align*}
        \widehat \beta_G^{V}\colon \cC&\too\cC\\
        X&\longmapsto (V\otimes X^{\otimes n})_{hG}.
    \end{align*}
\end{defn}

\begin{proposition}
\label{beta_polynomial}
    The functor $\widehat\beta_G^{V}$ is polynomial of degree $n$.
\end{proposition}
\begin{proof}
    Let $i\colon BG\to B\Sigma_n$ classify the action of $G$ on $[n]$. Then the functor $\widehat\beta_G^{V}$ factors as the composition
    \[\cC\xlongrightarrow{(-)^{\otimes n}}\cC^{B\Sigma_n}\xlongrightarrow{i^*}\cC^{BG}\xlongrightarrow{V\otimes(-)}\cC^{BG}\xlongrightarrow{BG_!}\cC. \]
    All functors except the first are exact and thus linear, and the first is polynomial of degree $n$ (\Cref{nth_power_polynomial}).
\end{proof}
By \Cref{polynomial_functoriality}, we can apply $K$-theory to the polynomial functors $\widehat\beta_G^V$ to obtain
\[\mdef{\widehat\beta_G^V:\uK(\cC)\too\uK(\cC)}.\]
We will relate these operations to the semiadditive power operations of \Cref{semiadditive_power_operations}. The first main ingredient will be a description of the ring structure on $K(\cC)$ in terms of polynomial functoriality.
\subsubsection{Multiplicative Structures on Algebraic $K$-Theory}
     We begin by giving a construction of the lax symmetric monoidal structure on algebraic $K$-theory using \Cref{polynomial_functoriality}. We work more generally in the setting of a product-preserving functor
     \[T\colon\Cat_{\perf}\too\Spc.\]
     For motivation, consider  $\cC,\cD\in\Cat_\perf$. The functor
    $\cC\times\cD\to\cC\otimes\cD$ is polynomial of degree $2$, and if we could evaluate $T$ on polynomial functors this would give a map $T(\cC)\times T(\cD)\to T(\cC\otimes\cD)$. We show that such functoriality indeed assembles into a lax symmetric monoidal structure on $T$:
\begin{theorem}
\label{lax_sm}
    Let $T:\Cat_{\perf}\to\Spc$ be product-preserving. Then an extension
\[\begin{tikzcd}[ampersand replacement=\&]
	{\Cat_{\perf}} \& \Spc \\
	{\Cat_{\perf}^\poly}
	\arrow["T", from=1-1, to=1-2]
	\arrow[from=1-1, to=2-1]
	\arrow[dashed, from=2-1, to=1-2]
\end{tikzcd}\]
     gives $T$ the structure of a lax symmetric monoidal functor.
\end{theorem}   

We first recall some facts about lax symmetric monoidal functors. For a symmetric monoidal category $\cU$ encoded as a product-preserving functor
\[\cU\colon\Span(\Fin)\too\Cat,\] 
denote the corresponding cocartesian fibration by
\[\cU^\otimes\too \Span(\Fin).\]
When the symmetric monoidal structure on $\cU$ is cartesian, we denote $\cU^\otimes$ by $\cU^\times$. A lax symmetric monoidal functor from $\cU$ to $\cV$ is a triangle 
\[\begin{tikzcd}[ampersand replacement=\&]
	{\cU^\otimes} \&\& {\cV^\otimes} \\
	\& {\Span(\Fin)}
	\arrow[from=1-1, to=1-3]
	\arrow[from=1-1, to=2-2]
	\arrow[from=1-3, to=2-2]
\end{tikzcd}\]
    such that the functor $\cU^\otimes\to\cV^\otimes$ preserves inert edges.

In the case that the symmetric monoidal structure on $\cV$ is cartesian, there is a simple description of lax symmetric monoidal functors into $\cV$.  
\begin{proposition}
\label{cartesian_sm}
    Suppose $\cU$ is a symmetric monoidal category and $\cV$ is a cartesian symmetric monoidal category. There is a functor $\cV^\times\to\cV$ such that postcomposition induces an equivalence
    \[\Fun^{\otimes\textup{-lax}}(\cU,\cV)\simeq\Fun^\times(\cU^\otimes,\cV),\]
    where $\Fun^\times(-,-)\subset\Fun(-,-)$ denotes the full subcategory of finite-product-preserving functors.
\end{proposition}

\begin{proof}This is by \cite[2.4.1]{HA} (translated to the setting of operads over $\Span(\Fin)$) and the characterization of $\mathcal O$-monoids from \cite[Proposition 2.3.3]{haugseng2023analyticmonadsinftyoperadslawvere}.
\end{proof}
\begin{rem}
   The functor
$\cV^\times\too\cV$ is given on objects by $(X_1,\ldots,X_n)\mapsto X_1\times\ldots\times X_n$.
\end{rem}
\begin{proof}[Proof of \Cref{lax_sm}]
    The inclusion $\Cat_\perf\subset\Cat$ is lax symmetric monoidal, where $\Cat$ is equipped with its cartesian symmetric monoidal structure. Under the equivalence of \Cref{cartesian_sm}, this inclusion corresponds to a product-preserving functor
    \[\Cat_\perf^\otimes\too\Cat.\]
    The image of $(\cC_1,\ldots,\cC_n)$ is $\cC_1\times\ldots\times\cC_n$, and the image of active edge $(\cC_1,\ldots,\cC_n)\to\cC_1\otimes\ldots\otimes\cC_n$ is the multilinear functor \[\cC_1\times\ldots\times\cC_n\to\cC_1\otimes\ldots\otimes\cC_n\qin\Cat.\]
    Since multilinear functors are polynomial (\cite[6.1.3]{HA}), the functor $\Cat_\perf^\otimes\to\Cat$ factors through the (non-full) subcategory $\Cat_\perf^\poly\subset\Cat$ to give a unique lift
\[\begin{tikzcd}
	& {\Cat_\perf^\poly} \\
	{\Cat_\perf^\otimes} & \Cat.
	\arrow[from=1-2, to=2-2]
	\arrow[dashed, from=2-1, to=1-2]
	\arrow[from=2-1, to=2-2]
\end{tikzcd}\]
The functor $\Cat^\poly_\perf\to\Cat$ creates products, so the lift $\Cat_\perf^\otimes\to\Cat^\poly_\perf$ is product-preserving. Then by postcomposing with $T$ (using polynomial functoriality), we obtain a product-preserving functor
\[\Cat^\otimes_\perf\too\Cat_\perf^\poly\too \Spc,\]
giving a lax symmetric monoidal structure on $T$.
\end{proof}
We next lift the lax symmetric monoidal structure on $T\colon \Cat_{\perf}\to \Spc$ to a lax symmetric monoidal structure on $\widetilde T\colon \Cat_{\perf}\to\CMon(\Spc)$.

Let $\cV\in\CAlg(\Pr)$. Then $\CMon(\cV)\simeq\CMon(\Spc)\otimes\cV$ is equipped with a symmetric monoidal structure such that the functor $\CMon(\cV)\to\cV$ is lax symmetric monoidal.
\begin{proposition}
\label{unique_lift}
    Suppose $\cU\in\CMon(\Cat)$ is semiadditive and $\cV\in\CAlg(\Pr)$. Then a product-preserving functor $T\colon \cU\to \cV$ admits a unique lift
\[\begin{tikzcd}[ampersand replacement=\&,sep=large]
	\& {\CMon(\cV)} \\
	\cU \& \cV
	\arrow[from=1-2, to=2-2]
	\arrow["{\widetilde T}", dashed, from=2-1, to=1-2]
	\arrow["T"', from=2-1, to=2-2]
\end{tikzcd}\]
Furthermore, a lax symmetric monoidal structure on $T$ lifts uniquely to a lax symmetric monoidal structure on $\widetilde T$.
\end{proposition}
\begin{proof}
    The first part of the proposition is the universal property of commutative monoids (\cite[Corollary 5.15]{spans}). It says that the forgetful functor $\CMon(\cV)\to\cV$ induces an equivalence
    \[\Fun^\times(\cU,\CMon(\cV))\iso \Fun^\times(\cU,\cV).\]
     
    Postcomposition with the lax symmetric monoidal functor $\CMon(\cV)\to\cV$ induces a lax symmetric monoidal functor on functor categories with respect to the Day convolution (\cite[Proposition 3.3]{bms}). Therefore there is a commutative square
\[\begin{tikzcd}
	{\Fun^\times(\cU,\CMon(\cV))} & {\Fun^\times(\cU,\cV)} \\
	{\Fun(\cU,\CMon(\cV))} & {\Fun(\cU,\cV)}
	\arrow["\sim", from=1-1, to=1-2]
	\arrow[hook', from=1-1, to=2-1]
	\arrow[hook', from=1-2, to=2-2]
	\arrow[from=2-1, to=2-2]
\end{tikzcd}\]
of lax symmetric monoidal functors 
in which the vertical arrows are fully faithful. Passing to commutative algebras (and the identification of commutative algebras in the Day convolution with lax symmetric monoidal functors) gives a square
\[\begin{tikzcd}
	{\Fun^{\times,\otimes\textup{-lax}}(\cU,\CMon(\cV))} & {\Fun^{\times,\otimes\textup{-lax}}(\cU,\cV)} \\
	{\Fun^{\otimes\textup{-lax}}(\cU,\CMon(\cV))} & {\Fun^{\otimes\textup{-lax}}(\cU,\cV).}
	\arrow["\sim", from=1-1, to=1-2]
	\arrow[hook', from=1-1, to=2-1]
	\arrow[hook', from=1-2, to=2-2]
	\arrow[from=2-1, to=2-2]
\end{tikzcd}\]
Given a product-preserving lax symmetric monoidal functor $\cU\to\cV$, it follows that it admits a unique lift to a product-preserving lax symmetric monoidal functor $\cU\to\CMon(\cV)$.
\end{proof}
Applying the above to the functor $\uK\colon\Cat_\perf\to\Spc$, we obtain a lax symmetric monoidal structure on $K\colon\Cat_\perf\to\CMon(\Spc)$. Since the lax symmetric monoidal structure on $K$-theory is unique (\cite[Corollary 1.6]{blumberg2015uniquenessmultiplicativecyclotomictrace}), we learn the following:
    
\begin{cor}
    \label{unique_equivariance} Let $\cC\in\Cat_\perf$.   The image of the $\Sigma_n$-equivariant multilinear functor
\[\cC\times\ldots\times\cC\too\cC\otimes\ldots\otimes \cC\] under $\uK$ (using polynomial functoriality) is the $\Sigma_n$-equivariant map
\[\uK(\cC)\times\ldots\times\uK(\cC)\too\uK(\cC\otimes\ldots\otimes\cC)\]
provided by the unique lax symmetric monoidal structure on $\uK$. 
\end{cor}
When $\cC\in\Cat_\perf$ is symmetric monoidal, $K(\cC)$ is a commutative algebra. We now use \Cref{unique_equivariance} to describe the total power operation on $K(\cC)$.
\begin{proposition}
\label{power_op_polynomial}
    Let $\cC\in\Cat_\perf$ be symmetric monoidal, so that $K(\cC)$ is a commutative algebra. Then there is a commutative triangle
\[\begin{tikzcd}[ampersand replacement=\&, column sep=huge]
	{\uK(\cC)} \& {\uK(\cC^{B\Sigma_n})} \\
	\& {\uK(\cC)^{B\Sigma_n}}
	\arrow["{\uK((-)^{\otimes n})}", from=1-1, to=1-2]
	\arrow["{\mathcal P_n}"', from=1-1, to=2-2]
	\arrow[from=1-2, to=2-2]
\end{tikzcd}\]
where the vertical map is the assembly map for the constant $B\Sigma_n$-limit.
\end{proposition}
\begin{proof}
    There is a diagram
\[\begin{tikzcd}[ampersand replacement=\&, column sep=large]
	{\uK(\cC)} \& {\uK((\cC^n)^{h\Sigma_n})} \& {\uK((\cC^{\otimes n})^{h\Sigma_n})} \& {\uK(\cC^{B\Sigma_n})} \\
	{\uK(\cC)} \& {(\uK(\cC)^n)^{h\Sigma_n}} \& {\uK(\cC^{\otimes n})^{h\Sigma_n}} \& {\uK(\cC)^{B\Sigma_n}}
	\arrow[from=1-1, to=1-2]
	\arrow["{\uK((-)^{\otimes n})}", curve={height=-24pt}, from=1-1, to=1-4]
	\arrow[equals, from=1-1, to=2-1]
	\arrow[from=1-2, to=1-3]
	\arrow[from=1-2, to=2-2]
	\arrow[from=1-3, to=1-4]
	\arrow[from=1-3, to=2-3]
	\arrow[from=1-4, to=2-4]
	\arrow[from=2-1, to=2-2]
	\arrow["{\mathcal P_n}"', curve={height=24pt}, from=2-1, to=2-4]
	\arrow[from=2-2, to=2-3]
	\arrow[from=2-3, to=2-4]
\end{tikzcd}\]
in which the vertical arrows are all assembly maps, and the bottom middle map is given by the lax symmetric monoidal structure on $\uK$. The left and right squares commute by naturality of the assembly maps. The middle square commutes by the rectangle
\[\begin{tikzcd}[ampersand replacement=\&]
	{\uK((\cC^n)^{h\Sigma_n})} \& {\uK((\cC^{\otimes n})^{h\Sigma_n})} \\
	{\uK(\cC^n)^{h\Sigma_n}} \& {\uK(\cC^{\otimes n})^{h\Sigma_n}} \\
	{(\uK(\cC)^n)^{h\Sigma_n}} \& {\uK(\cC^{\otimes n})^{h\Sigma_n}}
	\arrow[from=1-1, to=1-2]
	\arrow[from=1-1, to=2-1]
	\arrow[from=1-2, to=2-2]
	\arrow[from=2-1, to=2-2]
	\arrow["\wr"', from=2-1, to=3-1]
	\arrow["\wr", from=2-2, to=3-2]
	\arrow[from=3-1, to=3-2]
\end{tikzcd}\]
The top square in this rectangle commutes by naturality of the assembly transformation for the $B\Sigma_n$-limit along the functor $\uK\colon\Cat_\perf^\poly\to\Spc$. The bottom square commutes, as the middle and bottom horizontal arrows identify before taking $\Sigma_n$-fixed points (\Cref{unique_equivariance}).
\end{proof}
\subsection{Semiadditive Power Operations on Algebraic $K$-Theory}
In \cite{bms}, the authors introduce functors
\[\mdef{K^{[m]}\colon \Catstm\too\tsadi^{[m]}}\]
for every $0\le m\le\infty$, such that $K^{[m]}$ is the universal approximation of $K$-theory by a ($p$-typically) $m$-semiadditive functor (\cite[Section 6.3]{bms}). Here, \[\tsadi^{[m]}=\CMon_m^{(p)}(\Sp_{(p)})\] is the universal $p$-local stable and ($p$-typically) $m$-semiadditive presentable category. There is thus a comparison map $K\Rightarrow K^{[m]}$ on the level of underlying $K$-theory spectra (\cite[Corollary 6.11]{bms}).

For $\cC\in\CAlg(\Cat_\perf)$, the commutative algebra $K^{[1]}(\cC)$ is equipped with semiadditive power operations (\Cref{semiadditive_power_operations}). We now show that these operations are compatible with the operations on $\uK(\cC)$ coming from polynomial functoriality (\Cref{poly_operations_k_theory}).
\begin{theorem}
\label{ktheory_comparison}
   Let $\cC\in\CAlg(\Cat_\perf^{1\textup{-fin}})$. For a finite $p$-group $G$ acting on $[n]$ and $V\in\cC^{BG}$, there is a commutative square
\[\begin{tikzcd}[ampersand replacement=\&,sep=large]
	{\uK(\cC)} \& {\uK(\cC)} \\
	{\uK^{[1]}(\cC)} \& {\uK^{[1]}(\cC)}
	\arrow["{\widehat\beta_G^{V}}", from=1-1, to=1-2]
	\arrow[from=1-1, to=2-1]
	\arrow[from=1-2, to=2-2]
	\arrow["{\beta_G^{[V]}}", from=2-1, to=2-2]
\end{tikzcd}\]
where $[V]$ is the class of $V$ in $K^{[1]}(\cC)^{BG}\simeq K^{[1]}(\cC^{BG})$.
\end{theorem}
\begin{proof}
 Let $i\colon BG\to B\Sigma_n$ classify the action of $G$ on $[n]$. The operation $\widehat\beta_G^{V}$ on $\uK(\cC)$ is given by applying $\uK$ to the composition
     \[\cC\xlongrightarrow{(-)^{\otimes n}}\cC^{B\Sigma_n}\xlongrightarrow{i^*}\cC^{BG}\xlongrightarrow{V\otimes(-)}\cC^{BG}\xlongrightarrow{BG_!}\cC. \]
  Considering the last three maps in this composition, there is a commutative diagram
\[\begin{tikzcd}[ampersand replacement=\&,column sep=huge]
	{\uK(\cC^{B\Sigma_n})} \& {\uK(\cC^{BG})} \& {\uK(\cC^{BG})} \& {\uK(\cC)} \\
	{\uK^{[1]}(\cC^{B\Sigma_n})} \& {\uK^{[1]}(\cC^{BG})} \& {\uK^{[1]}(\cC^{BG})} \& {\uK^{[1]}(\cC)} \\
	{\uK^{[1]}(\cC)^{B\Sigma_n}} \& {\uK^{[1]}(\cC)^{BG}} \& {\uK^{[1]}(\cC)^{BG}} \& {\uK^{[1]}(\cC)}
	\arrow["{\uK(i^*)}", from=1-1, to=1-2]
	\arrow[from=1-1, to=2-1]
	\arrow["{\uK(V\otimes(-))}", from=1-2, to=1-3]
	\arrow[from=1-2, to=2-2]
	\arrow["{\uK(BG_!)}", from=1-3, to=1-4]
	\arrow[from=1-3, to=2-3]
	\arrow[from=1-4, to=2-4]
	\arrow["{\uK^{[1]}(i^*)}", from=2-1, to=2-2]
	\arrow["\wr"', from=2-1, to=3-1]
	\arrow["{\uK^{[1]}(V\otimes(-))}", from=2-2, to=2-3]
	\arrow["\wr"', from=2-2, to=3-2]
	\arrow["{\uK^{[1]}(BG_!)}", from=2-3, to=2-4]
	\arrow["\wr", from=2-3, to=3-3]
	\arrow[equals, from=2-4, to=3-4]
	\arrow["{i^*}", from=3-1, to=3-2]
	\arrow["{[V]\cdot(-)}", from=3-2, to=3-3]
	\arrow["{\int_{BG}(-)}", from=3-3, to=3-4]
\end{tikzcd}\]
The bottom vertical arrows are assembly maps, and the bottom right and bottom left squares commute because $K^{[1]}$ is a ($p$-typically) $1$-semiadditive functor. The top left and top right squares commute by naturality of the transformation $K\Rightarrow K^{[1]}$, and the middle column commutes because this is a symmetric monoidal transformation of lax symmetric monoidal functors (\cite[Corollary~6.19]{bms}). 

Next, considering the first map in the above composition, there is a commutative diagram
\[\begin{tikzcd}[ampersand replacement=\&, column sep=large]
	{\uK(\cC)} \& {\uK(\cC^{B\Sigma_n})} \\
	{\uK(\cC)} \& {\uK(\cC)^{B\Sigma_n}} \\
	{\uK^{[1]}(\cC)} \& {\uK^{[1]}(\cC)^{B\Sigma_n}}
	\arrow["{\uK((-)^{\otimes n})}", from=1-1, to=1-2]
	\arrow[equals, from=1-1, to=2-1]
	\arrow[from=1-2, to=2-2]
	\arrow["{\mathcal P_n}", from=2-1, to=2-2]
	\arrow[from=2-1, to=3-1]
	\arrow[from=2-2, to=3-2]
	\arrow["{\mathcal P_n}", from=3-1, to=3-2]
\end{tikzcd}\]
in which the top square commutes by \Cref{power_op_polynomial} and the bottom square commutes since the comparison $K(\cC)\to K^{[1]}(\cC)$ is a map of commutative algebras. The right vertical composition is isomorphic to 
\[\uK(\cC^{B\Sigma_n})\too \uK^{[1]}(\cC^{B\Sigma_n})\too\uK^{[1]}(\cC)^{B\Sigma_n}\]
by naturality of the assembly map. Pasting the two diagrams gives the result.
\end{proof}
\begin{war}
    The functor $K^{[m]}$ is \textit{not} an $\ell$-typically $m$-semiadditive functor for $\ell\neq p$. In particular, the conclusion of \Cref{ktheory_comparison} does not hold when $G$ is not a $p$-group.
\end{war}
\begin{corollary}
\label{lnf_comparison}
    Suppose $\cC\in\CAlg(\Cat_\perf^{1\textup{-fin}})$ is $L_n^f$-local. Then the conclusion of \Cref{ktheory_comparison} holds with $\underline K^{[1]}(\cC)$ replaced by $\uK_{T(n+1)}(\cC)$.
\end{corollary}
\begin{proof}
Since $\cC$ is $L_n^f$-local, the image of the comparison map $K(\cC)\to K^{[1]}(\cC)$ under $T(n+1)$-localization is an isomorphism 
    \[K_{T(n+1)}(\cC)\iso K^{[1]}_{T(n+1)}(\cC)\]
     by \cite[Corollary 3.7]{descent}.
     
    Since $\tsadi^{[1]}\to \Sp_{T(n+1)}$ is a symmetric monoidal localization (\cite[Lemma 6.12]{bms}), the localization
    \[K^{[1]}(\cC)\too K^{[1]}_{T(n+1)}(\cC)\]
    is a map of commutative algebras in the higher semiadditive category $\tsadi^{[1]}$. By naturality of the operation $\beta_G^{[V]}$, there is thus a commutative diagram
\[\begin{tikzcd}[ampersand replacement=\&]
	{\uK^{[1]}(\cC)} \& {\uK^{[1]}(\cC)} \\
	{\uK^{[1]}_{T(n+1)}(\cC)} \& {\uK_{T(n+1)}^{[1]}(\cC)}
	\arrow["{\beta_G^{[V]}}", from=1-1, to=1-2]
	\arrow[from=1-1, to=2-1]
	\arrow[from=1-2, to=2-2]
	\arrow["{\beta_G^{[V]}}", from=2-1, to=2-2]
\end{tikzcd}\]
    Pasting this square with the square of \Cref{ktheory_comparison} and then using the above identification of $\uK^{[1]}_{T(n+1)}(\cC)$ with $ \uK_{T(n+1)}(\cC)$ gives the result.
\end{proof}

\section{Power Operations in Semiadditive Height 1}
In this section we study semiadditive power operations in height $1$. $\delta$-rings and $\lambda$-rings are essential, so we recall them in \Cref{delta_rings} and prove some results about $\delta$-rings in the relative setting. We then study the semiadditive power operation $\theta$ (\Cref{theta_details}) and use it to classify power operations assuming virtual orientability (\Cref{proof}). We conclude the section by relating these results to the redshift philosophy (\Cref{lambda_redshift}).

\subsection{$\delta$-rings and $\lambda$-rings}

\label{delta_rings}
\subsubsection{$\delta$-rings}
In this section we show some basic facts about $\delta$-rings in the relative setting.
\begin{defn}
\label{pderivation}
    A \mdef{\textit{$p$-derivation}} on a (discrete) commutative ring $R$ is a function
    \[\delta\colon R\too R\]
    that satisfies
    \begin{enumerate}
        \item (additivity) $\delta(x+y)=\delta(x)+\delta(y)+\frac{x^p+y^p-(x+y)^p}{p}$ for all $x,y\in R$.
        \item (multiplicativity) $\delta(xy)=\delta(x)y^p+x^p\delta(y)+p\delta(x)\delta(y)$.
        \item (normalization) $\delta(0)=\delta(1)=0$.
    \end{enumerate}
    The pair $(R,\delta)$ is called a \mdef{\textit{$\delta$-ring}}. Denote by $\mdef{\CRing^\delta}$ the $1$-category of $\delta$-rings and $\delta$-ring homomorphisms. 
    
    An \mdef{\textit{additive $p$-derivation}} on $R$ is a function $\delta\colon R\to R$ which satisfies properties (1) and (3).
\end{defn}
\begin{rem}\label{frobenius_lift}
The above identities are engineered so that $\psi(x)=x^p+p\delta(x)$ is a ring homomorphism that lifts the Frobenius endomorphism on $R/p$. When $R$ is torsion-free, the operation $\delta$ is entirely determined by the Frobenius lift $\psi$. In general, a $p$-derivation on $R$ should be considered the data of a \textit{derived} lift $\psi$ of Frobenius, so that $\delta$ `witnesses' the fact that $\psi$ and $x^p$ agree on $R/p$ (\cite[Remark 2.5]{scholze}).
\end{rem}

  \begin{rem}
\label{forgetful_adjoints}
 The forgetful functor $U\colon \CRing^\delta\to\CRing$ admits both left and right adjoints. The right adjoint can be identified with the functor $W_p$ of $p$-typical Witt vectors (\cite{joyal_delta}), and the left adjoint satisfies
\[\Free^\delta(\ZZ[t])=\ZZ[t,\delta(t),\delta^2(t),\ldots],\]
with the action of $\delta$ extended from the generators (for example, by \cite[Lemma 2.11]{scholze}).
\end{rem}

\begin{defn}
Let $R$ be a $\delta$-ring. The $1$-category of $\delta$-$R$-algebras is the under category
\[\mdef{\CRing_R^\delta:=\CRing^\delta_{R/}}.\]
\end{defn}
The forgetful functor $U$ induces the `underlying $R$-algebra' functor $U_R\colon \CRing_R^\delta\to \CRing_R$. 

\begin{proposition}
    The functor $U_R\colon \CRing_R^\delta\to\CRing_R$ has both left and right adjoints. The left adjoint satisfies
\[\Free^\delta_R(R[t])=R[t,\delta(t),\delta^2(t),\ldots],\]
with the action of $\delta$ extended from the generators. 
\end{proposition}

\begin{proof}
The functor $U\colon \CRing^\delta\to\CRing$ admits both right and left adjoints (\Cref{forgetful_adjoints}). Passing to slice categories gives a diagram of adjunctions
\[\begin{tikzcd}[ampersand replacement=\&]
	{\CRing^\delta_{R/}} \& {\CRing_{U(R)/}}
	\arrow[""{name=0, anchor=center, inner sep=0}, from=1-1, to=1-2]
	\arrow[""{name=1, anchor=center, inner sep=0}, "{\Free^\delta_R}"', curve={height=18pt}, from=1-2, to=1-1]
	\arrow[""{name=2, anchor=center, inner sep=0}, "{W_p}", curve={height=-18pt}, from=1-2, to=1-1]
	\arrow["\perp"{description}, draw=none, from=0, to=2]
	\arrow["\perp"{description}, draw=none, from=1, to=0]
\end{tikzcd}\]
To identify $\Free_R^\delta(R[t])$, observe that there is a commutative diagram of left adjoints
\[\begin{tikzcd}[ampersand replacement=\&]
	{\CRing^\delta} \& {\CRing^\delta_R} \\
	\CRing \& {\CRing_{U(R)}}
	\arrow["{-\otimes R}", from=1-1, to=1-2]
	\arrow["{\Free^\delta}", from=2-1, to=1-1]
	\arrow["{-\otimes R}", from=2-1, to=2-2]
	\arrow["{\Free^\delta_R}"', from=2-2, to=1-2]
\end{tikzcd}\]
Thus $\Free_R^\delta(R[t])$ is the tensor product $R\otimes\Free^\delta(\ZZ[t])$ of $\delta$-rings. Since the tensor product of $\delta$-rings is computed as the tensor product of underlying commutative rings, the result follows.
\end{proof}
\begin{defn}
\label{delta_weight}
We denote
\[\mdef{\Free_R^\delta(t):=\Free_R^\delta(R[t])}.\]
Since $\Free_R^\delta(t)$ is the free polynomial algebra on generators $\delta^r(t)$, we define a grading on it by setting $|\delta^r(t)|=p^r$. We refer to this as the \mdef{\textit{weight grading}} on a free $\delta$-ring.
\end{defn}
\begin{example}
\label{Deltacoproducts}
For every set $T$,
\[\Free_R^\delta(T)=\bigotimes_{t\in T}\Free_R^\delta(t)\]
is the polynomial algebra on iterates of the operation $\delta$ applied to the elements of $T$. 
\end{example}
Since it is free, $\Free_R^\delta(t)$ has lots of additional structure. In particular, it is equipped with a coaddition
\begin{align*}\Delta\colon \Free_R^\delta(t)\too \Free_R^\delta(t)\otimes \Free_R^\delta(t)
\end{align*}
that represents addition on $\delta$-$R$-algebras. $\Delta$ is determined by
\[\Delta(t)=t\otimes 1+1\otimes t.\]
A primitive element $\alpha\in \Free_R^\delta(t)$ corresponds to a unary operation on $\delta$-$R$-algebras that satisfies the equation $\alpha(x+y)=\alpha(x)+\alpha(y).$ Since the Frobenius lift $\psi(t)=t^p+p\delta(t)$ is such an additive operation, it corresponds to a primitive element in $\Free_R^\delta(t)$. In fact, $\psi$ and its iterates are the only ones:
\begin{proposition}
\label{primitive_classification}
    Suppose $\alpha\in \Free_R^\delta(t)$ is a primitive element of pure weight $n$. Then 
    \[\alpha =c\cdot \psi^{ r}(t),\]
    where $c\in R$ and $\psi$ is the Frobenius lift associated to $\delta$.
\end{proposition}
\begin{proof}
    The coaddition $\Delta$ on $\Free_R^\delta(t)$ is given by tensoring the coproduct
    \[\Delta\colon \Free_{\ZZ}^\delta(t)\too\Free_{\ZZ}^\delta(t)\otimes\Free_{\ZZ}^\delta(t)\]
    with $R$, so it is enough to classify primitives in the universal case.  Since both sides are free $\ZZ$-modules, we can read off the primitives after rationalization. Upon rationalization, \[\Free_{\ZZ}^\delta(t)\otimes\QQ\cong\QQ[t,\psi(t),\psi^2(t),\ldots]\]
    is a polynomial algebra on the iterates of the Frobenius lift $\psi$ (e.g., by \cite[Lemma 2.11]{scholze}). Since each $\psi^r(t)$ is primitive and a product of primitives is not primitive, the result follows.
\end{proof}

\begin{rem}
\label{graded_delta}
    We will require a variant of the above which accounts for \textit{graded} rings $R_*$ whose $0$th graded piece $
    R_0$ is a $\delta$-ring. In this case, we define
    \[\mdef{\Free_{R_*}^\delta(t):=\Free_{R_0}^\delta(t)\otimes_{R_0}R_*}.\]
    The weight grading on $\Free_{R_0}^\delta(t)$ extends naturally to $\Free_{R_*}^\delta(t)$, making it a \textit{bigraded} ring. Also, $\Free_{R_*}^\delta(t)$ has a coaddition $\Delta$ given by tensoring the coaddition on $\Free_{R_0}^\delta(t)$ with $R_*$. 
\end{rem}
\subsubsection{$\lambda$-rings}
In this section we recall different definitions of $\lambda$-rings and describe how to pass between them. The results are summarized in \Cref{tab:my_table}.
\begin{defn}[{\cite{knutson}}]
\label{defn1}A \mdef{\textit{$\lambda$-ring}} is a commutative ring $R$ equipped with functions $\mdef{\lambda^n\colon R\to R}$, for $n\ge0$, that satisfy the identities
\begin{enumerate}
    \item $\lambda^0(x)=1$ and $\lambda^1(x)=x$,
    \item $\lambda^n(x+y)=\sum_{i=0}^n\lambda^i(x)\lambda^{n-i}(y)$,
    \item $\lambda^n(xy)=P_n(\lambda^1(x),\ldots,\lambda^n(x),\lambda^1(y),\ldots,\lambda^n(y))$,
    \item $\lambda^m(\lambda^n(x))=P_{m,n}(\lambda^1(x),\ldots,\lambda^{mn}(x))$,
\end{enumerate}
where the $P_n,P_{m,n}$ are certain integer polynomials.
\end{defn}
\Cref{defn1} axiomatizes the action of exterior powers on vector spaces, so that the operation $\lambda^n$ is modeled on the $n$th exterior power. There is an alternative, equivalent definition in which the operations are modeled on the \textit{symmetric} powers of vector spaces.
\begin{defn}[{\cite{tau_rings}}]
\label{defn2}
    A \textit{\mdef{$\lambda$-ring}} is a commutative ring $R$ equipped with functions $\mdef{\sigma^n\colon R\to R}$, for $n\ge0$, satisfying conditions analogous to those of \Cref{defn1}.
\end{defn}
There is a \textit{third} equivalent definition that dispenses with the inexplicit polynomials $P_n,P_{m,n}$ that appear in the previous two definitions.
\begin{defn}[{\cite{joyal_lambda,rezk_lambda,hubner}}]
\label{defn3}
A \mdef{\textit{$\lambda$-ring}} is a commutative ring $R$ equipped with $p$-derivations $\mdef{\delta_p\colon R\to R}$ for all primes $p$, such that
\[\psi_p\circ\delta_q=\delta_q\circ \psi_p,\]
where $\psi_p(x)=x^p+p\delta_p(x)$ is the Frobenius lift associated to $\delta_p$.
\end{defn}
\begin{rem}
\label{ptypical}
When $R$ is $p$-local, an $\ell$-derivation (for $\ell\neq p$) is nothing more than a ring endomorphism of $R$ (\Cref{frobenius_lift}). Therefore, we consider a $\delta$-ring to be the `$p$-typical' analog of a $\lambda$-ring, in the sense that it is the interesting part of a $\lambda$-ring structure on a $p$-local ring (\cite[Remark~2.3]{scholze}).
\end{rem}
All three of the above definitions give a `generators-and-relations' presentation of a $\lambda$-ring. To move between them, we use the following:

\begin{theorem}[{\cite{knutson,donald}}]
\label{free_lambda}
    The ring
    \[\mdef{\Lambda:=\bigoplus_{n\ge0}\Rep_\CC^\textup{f.d.}(\Sigma_n)}\]
    is the free $\lambda$-ring on a single element. The Frobenius characteristic map gives an isomorphism \[\Lambda\iso \textup{Symm}\]
    with the ring of symmetric functions.
\end{theorem}
Every element of $\Lambda$ defines a natural operation on $\lambda$-rings. We now identify the operations appearing in the above definitions.
\begin{example}
\label{power_character}
    The identity operation $x\mapsto x$ corresponds to the trivial $1$-dimensional representation of $\Sigma_1$ in $\Lambda$. Since the ring structure on $\Lambda$ is the `outer/induction product' (\cite[III.1]{knutson}), it follows that the operation $x\mapsto x^n$ corresponds to the regular representation of $\Sigma_n$. 
\end{example}
\begin{example}
    The operation $\lambda^n$ corresponds to the sign representation of $\Sigma_n$. Under the isomorphism $\Lambda\cong \text{Symm}$, the sign representation corresponds to the $n$th elementary symmetric polynomial $e_n$. 
    \end{example}\begin{example}
    The operation $\sigma^n$ corresponds to the trivial representation of $\Sigma_n$, which corresponds to the $n$th homogeneous symmetric polynomial $h_n$. 
    \end{example}
    By symmetric function identities, it follows that $\lambda^n$ and $\sigma^n$ are related by
    \[\sum_{i=0}^n(-1)^i\sigma^n(x)\lambda^{n-i}(x)=0.\]
\begin{example}
\label{psi_character}
    The Frobenius lift $\psi_p(x)=x^p+p\delta_p(x)$ corresponds to the $p$th power sum polynomial under the isomorphism $\Lambda\cong\text{Symm}$. The virtual character of the corresponding $\Sigma_p$-representation is equal to $p$ when evaluated on $p$-cycles in $\Sigma_p$, and $0$ otherwise (\cite[III.2]{knutson}, \cite[18.6]{haze}).
\end{example}
We now determine the virtual representations in $\Lambda$ corresponding to the operations $\delta_p$ appearing in \Cref{defn3}.
\begin{proposition}
\label{delta_character}
    Let $V_p\in\Rep_\CC^\textup{f.d.}(\Sigma_p)$ be the virtual $\Sigma_p$-representation corresponding to the $\lambda$-ring operation $\delta_p$. Then the character of $V_p$ is given by
    \begin{align*}
       \chi_{V_p}\colon  [\sigma]&\longmapsto\begin{cases}1&\sigma\textup{ is a }p\textup{-cycle}\\
        -(p-1)!&\sigma\textup{ is the identity}\\
        0&\textup{otherwise}\end{cases}
    \end{align*}
\end{proposition}
\begin{proof}
    The character follows from the equation $\psi_p(x)=x^p+p\delta_p(x)$, using the characters of the representations corresponding to the operations $x^p$ and $\delta_p$ (\Cref{power_character}, \Cref{psi_character}).
\end{proof}
We can use \Cref{delta_character} to give another expression for $V_p$.
\begin{corollary}
\label{delta_rep}
    There is an equality of virtual representations
    \[V_p=\frac{1}{p-1}\bigg(\CC[\Sigma_p/C_p]-\CC[\Sigma_p]\bigg)\]
    in $\Rep_\CC^\textup{f.d.}(\Sigma_p)$.
\end{corollary}
\begin{proof}
The character of $\CC[\Sigma_p/C_p]$ is
\[\chi_{\CC[\Sigma_p/C_p]}\colon[\sigma]\longmapsto\begin{cases}|N_{\Sigma_p}(C_p)/C_p|=p-1&\sigma\text{ is a }p\text{-cycle}\\
(p-1)!&\sigma\text{ is the identity}\\
0&\text{otherwise}\end{cases}\]
The result follows by comparing the characters of both sides.
\end{proof}
\begin{rem}
   Let $\rho$ be a nontrivial irreducible representation of $C_p$. Then a character computation shows that $-V_p$ is also equal to the representation $\Ind_{C_p}^{\Sigma_p}\rho$.
\end{rem}

\begin{table}[htbp]
    \centering
    \caption{$\lambda$-ring operations, their symmetric group representations, and their symmetric functions}
    \label{tab:my_table}
    
    \renewcommand{\arraystretch}{1.5} 
    \setlength{\tabcolsep}{12pt}      
    
    \begin{tabular}{c|c|c}
        
        Operation &(Virtual) $\Sigma_n$-Representation & Symmetric Function \\
        \hline
        $x^n$ & $\CC[\Sigma_n]$ & $(x_1+x_2+\ldots)^n$ \\
        \hline
        $\sigma^n$& triv & $h_n$ \\
        \hline
        $\lambda^n$ & sign & $e_n$ \\
        \hline
        $\delta_p$ & $V_p$&$\frac{(x_1^p+x_2^p+\ldots)-(x_1+x_2+\ldots)^p}{p}$\\
        \hline
        $\psi_p$& $\CC[\Sigma_p]+pV_p$&$x_1^p+x_2^p+\ldots$
    \end{tabular}
\end{table}
 
\subsection{The Semiadditive Power Operation $\theta$}
\label{theta_details}
In this subsection we recall the operation $\theta$ and study its properties.

\begin{defn}
\label{theta}
    Let $\cC\in\CAlg(\Cat_\st)$ be $p$-complete and $1$-semiadditive. For $R\in\CAlg(\cC)$, define the operation
    \[\mdef{\theta\colon\underline R\too\underline R}\]
 by the formula
    \[\theta(x):=\frac{1}{p|BC_p|-1}\bigg(-|BC_p|\cdot x^p+\int_{BC_p}x^p\bigg).\]
    Note that the denominator $p|BC_p|-1$ is a unit, since $\cC$ is $p$-complete.
\end{defn}
\begin{rem}
    This semiadditive description of the operation $\theta$ appears in \cite[Proposition~5.4]{gwitt}, where the authors also verify that it is the operation defined by Hopkins in \cite{hopkins}. Note that their formula assumes $p|BC_p|=p$, which holds in chromatic height $1$ by explicit computation. The formula of \Cref{theta} is the correct one in general.
\end{rem}
\begin{rem}
    By a `partition of unity' in the $p$-complete Burnside ring of $\Sigma_p$ (\cite[Theorem~5.2.5]{yifan}), we have
    \[\int_{B\Sigma_p}x^p=\frac{1-(p-2)!}{p!}x^p+\frac{1}{p-1}\int_{BC_p}x^p.\]
    It follows that $\theta$ can be expressed in terms of integration over $B\Sigma_p$ as follows:
    \[\theta(x)=\frac{(p-1)!}{p!|B\Sigma_p|-1}\bigg(-|B\Sigma_p|\cdot x^p+\int_{B\Sigma_p}x^p\bigg).\]
\end{rem}

\begin{lem}
\label{theta_additive}
    The operation $\theta$ defines an additive $p$-derivation on $\pi_0R$.
\end{lem}
\begin{proof}
    The normalization condition of \Cref{pderivation} is immediate. The additivity property of $\theta$ follows from the additivity property of the operation $\alpha$ (\Cref{alpha}). 
\end{proof}

\subsubsection{Multiplicativity of $\theta$}
Fix $R\in\CAlg(\cC)$. We now consider when $\theta$ satisfies the multiplicativity identity of \Cref{pderivation}. Let 
\[e\colon\pt\to B\Sigma_p, \ \ q\colon B\Sigma_p\to\pt,\]
and define $\mdef{\overline\theta\colon R^{B\Sigma_p}\to R}$ by the formula
\[\overline\theta=\frac{(p-1)!}{p!|B\Sigma_p|-1}\big(-|B\Sigma_p|e^*+q_*\big).\]
By construction, we have that $\theta=\overline\theta\circ\mathcal P_p.$
\begin{lem}
\label{splitting}
   Let $\cC\in\CAlg(\Cat^{1\textup{-}\oplus})$ be $p$-complete. For $R\in \CAlg(\cC)$, the map
    \[e_*\oplus q^*\colon R\oplus R\too R^{B\Sigma_p}\]
    exhibits $R\oplus R$ as a retract of $R^{B\Sigma_p}$.
\end{lem}
\begin{proof}
    The composition
\[\begin{tikzcd}[ampersand replacement=\&,column sep =large]
	{ R\oplus R} \& { R^{B\Sigma_p}} \& { R\oplus R}
	\arrow["{e_*\oplus q^*}", from=1-1, to=1-2]
	\arrow["{e^*\oplus q_*}", from=1-2, to=1-3]
\end{tikzcd}\]
    is given by a matrix with determinant
    \[\det\begin{bmatrix}p!&1\\
    1&p|B\Sigma_p|\end{bmatrix}=p!|B\Sigma_p|-1.\]
Since $\cC$ is $p$-complete, this matrix is invertible.
\end{proof}
\begin{proposition}
\label{degenerate}
    Let $\cC\in\CAlg(\Cat_\st^{1\textup{-}\oplus})$ be $p$-complete. Let $f,g\in\pi_* R^{B\Sigma_p}$, and suppose $f$ is contained in the retract $R\oplus R\subset R^{B\Sigma_p}$ (\Cref{splitting}). Then    \[\overline\theta(fg)=e^*(f)\overline\theta(g)+\overline\theta(f)e^*(g)+p\overline\theta(f)\overline\theta(g).\]
\end{proposition}
\begin{proof}
Define the function
    \begin{align*}
        \langle-,-\rangle_\theta\colon\pi_* R^{B\Sigma_p}\times\pi_* R^{B\Sigma_p}\too\pi_* R
    \end{align*}
    by the formula
    \[\langle f,g\rangle_\theta=\overline\theta(fg)-e^*(f)\overline\theta(g)-\overline\theta(f)e^*(g)-p\overline\theta(f)\overline\theta(g).\]
    On $\pi_0$, the function $\langle-,-\rangle_\theta$ measures the failure of $\theta$ to satisfy the multiplicativity identity of a $p$-derivation. Since it is linear in both variables, to prove the claim it is enough to show $\langle f,g\rangle_\theta=0$ when $f$ is in the image of $e_*$ or $q^*$.
    \begin{enumerate}
        \item If $f=e_*(x)$ for $x\in\pi_* R$, then
        \begin{align*}
            \overline\theta(e_*(x))&=\frac{(p-1)!}{p!|B\Sigma_p|-1}\big(|B\Sigma_p|e^*(e_*(x))-q_*(e_*(x))\big)\\
            &=-x(p-1)!.
        \end{align*}
        For any $g\in\pi_* R^{B\Sigma_p}$ we also have
        \begin{align*}        
        \overline\theta(e_*(x)g)&=\frac{(p-1)!}{p!|B\Sigma_p|-1}\big(|B\Sigma_p|e^*(e_*(x)g)-q_*(e_*(x)g)\big)\\
        &=-xe^*(g)(p-1)!
        \end{align*}
        because $e^*$ takes products to products and because
        \[q_*(e_*(x)g)=xe^*(g)\]
       by \Cref{dirac}.
        \item If $f=q^*(x)$ for $x\in\pi_* R$, then
        \begin{align*}
            \overline\theta(q^*(x))&=\frac{(p-1)!}{p!|B\Sigma_p|-1}\big(|B\Sigma_p|e^*(q^*(x))-q_*(q^*(x))\big)\\
            &=0.
        \end{align*}
        For any $g\in\pi_* R^{B\Sigma_p}$, we also have $\overline\theta(q^*(x)g)=x\overline\theta(g)$ because $\overline\theta$ is linear.
    \end{enumerate}
    With the above, it becomes a straightforward verification that $\langle f,g\rangle_\theta=0$ when either $f=e_*(x)$ or $f=q^*(x)$.
\end{proof}

We now relate the retract of \Cref{splitting} to the virtual orientability of $R$. To do so, we first recall the following fact.
\begin{proposition}[{\cite[Theorem B]{fusion}}]
   The inclusion $C_p\rtimes\FF_p^\times\subset\Sigma_p$ induces an equivalence
    \[B(C_p\rtimes\FF_p^\times)_p^\wedge\iso  (B\Sigma_p)_p^\wedge\]
    on classifying spaces upon $p$-completion.
\end{proposition}
\begin{proposition}
\label{virtual}
    Let $\cC\in\CAlg(\Cat^{1\textup{-}\oplus})$. Then if $R\in\CAlg(\cC)$ is virtually orientable of height $1$, the inclusion of \Cref{splitting} is an isomorphism $ R\oplus R\simeq R^{B\Sigma_p}$.
\end{proposition}
\begin{proof}
 
    We can check that
    $e_*\oplus q^*\colon  R\oplus R\to R^{B\Sigma_p}$
    is an isomorphism after faithful extension of scalars, so we assume that $\cC$ is orientable with invertible Fourier transform $\mathfrak F_\omega$. Unwinding definitions, the composition
    \begin{align*}
         R\oplus R\too R^{B\Sigma_p}&\iso ( R^{BC_p})^{h\FF_p^\times}\\
        &\iso ( R[BC_p])_{h\FF_p}\\
        &\xlongrightarrow{\mathfrak F_\omega}( R^{C_p})_{h\FF_p^\times}\\
        &\iso  R[B\FF_p^\times]\oplus R\\
        &\iso  R\oplus R
    \end{align*}
    is given by a matrix with determinant
    \[\det\begin{bmatrix}
        (p-1)!&|BC_p|\\
        (p-1)(p-1)!&0
    \end{bmatrix}=(p-1)(p-1)!|BC_p|.\]
    Since $R$ has semiadditive height 1, $|BC_p|$ is a unit and this matrix is invertible. Therefore $e_*\oplus q^*$ is an isomorphism.
\end{proof}
\begin{rem}
    \label{basis}
   The composition
\[\begin{tikzcd}[ampersand replacement=\&]
	{ R\oplus R} \& { R^{B\Sigma_p}} \& { R\oplus R}
	\arrow["{e_*\oplus q^*}", from=1-1, to=1-2]
	\arrow["{e^*\oplus\overline\theta}", from=1-2, to=1-3]
\end{tikzcd}\]
is similarly given by an invertible matrix. Since the first map is an isomorphism, so is the second.
\end{rem}
We deduce the following:
\begin{theorem}
\label{theoremAproof}
   Let $\cC\in\CAlg(\Cat_\st^{1\textup{-}\oplus})$. If $R\in\CAlg(\cC)$ is virtually orientable of height $1$, then $\theta:\pi_0R\to\pi_0R$ is a $p$-derivation.
\end{theorem}
\begin{proof}
The operation $\theta$ is always an additive $p$-derivation (\Cref{theta_additive}), so it only remains to verify the multiplicativity identity. The retract of \Cref{splitting} is an equivalence by \Cref{virtual}, so for every pair $x,y\in\pi_0R$, we have
\begin{align*}
    \underbrace{\overline\theta(\mathcal P_p(x)\mathcal P_p(y))}_{\overline\theta(\mathcal P_p(xy))}-\overline\theta(\mathcal P_p(x))y^p-x^p\overline\theta(\mathcal P_p(y))-p\overline\theta(\mathcal P_p(x))\overline\theta(\mathcal P_p(y))=0
\end{align*}
by \Cref{degenerate}. Since $\theta(-)=\overline\theta(\mathcal P_p(-))$, the result follows.
\end{proof}

\subsection{Classifying Power Operations in Semiadditive Height $1$}
\label{proof}
We briefly introduce the theory of power operations on commutative algebras in a symmetric monoidal category $\cC\in\CAlg(\Pr_\st)$ (see \cite{powers}). The free commutative algebra on an object $X\in\cC$ is given by
\[\Free^{\mathbb E_\infty}_{\cC}(X)=\bigoplus_{n\ge0}X^{\otimes n}_{h\Sigma_n}.\]
We denote $\mdef{\Free^{\mathbb E_\infty}_{\cC}(X)^{\wt=n}:=X^{\otimes n}_{h\Sigma_n}}$. For a set $T$, we denote
$\mdef{\Free^{\mathbb E_\infty}_{\cC}(T):=\Free^{\mathbb E_\infty}_{\cC}(\one_\cC^{\oplus T})}.$ 

Since $\cC$ is stable, every object $X\in\cC$ has an \textit{underlying spectrum} $\textup{map}_\cC(\one_\cC,X)$, so that $X$ has $\ZZ$-graded homotopy groups. If $R\in\CAlg(\cC)$ then $\pi_*R$ is a graded-commutative ring.
\begin{defn}
    A \mdef{\textit{power operation of degree $k$}} on commutative algebras in $\cC$ is a natural transformation $\pi_0\Rightarrow \pi_k$ of functors $\CAlg(\cC)\too\Set$.\footnote{We restrict our attention to operations $\pi_0\Rightarrow\pi_k$, although one could similarly consider operations that act on classes in nonzero degrees.}
    \end{defn}

      The following is an immediate consequence of the Yoneda lemma:
    \begin{proposition}[{\cite[Corollary 4.12]{powers}}]
        Power operations of degree $k$ are in natural bijection with $\pi_k\Free^{\mathbb E_\infty}_{\cC}(t)$.
    \end{proposition}
\begin{rem}
   More generally, $\pi_k\Free^{\mathbb E_\infty}_{\cC}(T)$ classifies $T$-ary operations
    \[\underbrace{\pi_0R\times\ldots\times\pi_0R}_{T\textup{ times}}\too\pi_kR\]
    that are natural in $R\in\CAlg(\cC)$.
\end{rem}
\begin{defn}
\label{weight}
    If a class $\alpha\in\pi_*\Free^{\mathbb E_\infty}_{\cC}(T)$ lies in the summand $\pi_*(\Free^{\mathbb E_\infty}_{\cC}(T)^{\wt=n})$, then we say $\alpha$ is an operation of \mdef{\textit{weight $n$.}}
\end{defn}
If the functor $\pi_*$ commutes with infinite direct sums (e.g. if $\one_\cC$ is compact), then every class in $\pi_*\Free^{\mathbb E_\infty}_{\cC}(t)$ decomposes into a finite sum of operations of varying weight (\cite[Remark 4.13]{powers}). In the general case we restrict attention to operations with only finitely many nonzero weights.
\begin{defn}\label{uncompleted_operations}The subring\[\bigoplus_{n\ge0}\pi_*\Free^{\mathbb E_\infty}_{\cC}(T)^{\wt=n}\subset\pi_*\Free^{\mathbb E_\infty}_{\cC}(T)\]
    is the ring of \textit{\mdef{bounded weight \textit{($T$-ary) power operations}}}.
    \end{defn}
    \subsubsection{Power Operations over Virtually Orientable Rings of Height $1$}
Let $\cC\in\CAlg(\Pr_{\st}^{1\textup{-}\oplus})$, and suppose $R\in\CAlg(\cC)$ is virtually orientable of height $1$. Denote $R_*=\pi_*R$. Since $R_0$ is a $\delta$-ring by \Cref{theoremAproof},
there is a canonical ring homomorphism
\[\mdef{\Theta\colon \Free^{\delta}_{R_*}(t)\too \pi_*(\Free^{\mathbb E_\infty}_{R}(t))}\]
taking $t\mapsto t$. More generally, there is a canonical homomorphism $\Theta\colon \Free^{\delta}_{R_*}(T)\to \pi_*(\Free^{\mathbb E_\infty}_{R}(T))$ for any set $T$. This homomorphism $\Theta$ respects the weight gradings on the source and target (\Cref{delta_weight}, \Cref{weight}).
\begin{theorem}
\label{theoremBproof}
    Let $\cC\in\CAlg(\Pr_{\st}^{1\textup{-}\oplus})$, and suppose $R\in\CAlg(\cC)$ is virtually orientable of height $1$. For any finite set $T$, the canonical homomorphism
    \[\Theta\colon \Free^{\delta}_{R_*}(T)\too\pi_*(\Free^{\mathbb E_\infty}_{R}(T))\]
    exhibits the source as the subring of bounded weight $T$-ary power operations (\Cref{uncompleted_operations}).
\end{theorem}

We now outline our strategy. 
\begin{itemize}
\item First, note that it suffices to work in the universal case, namely $\cC=\widehat\tsadi_1$. This is useful because the underlying spectrum functor $\widehat\tsadi_1\to\Sp$ is conservative (as in \cite[Corollary 5.3.10]{ambiheight}), so isomorphisms in $\widehat\tsadi_1$ are detected on $\ZZ$-graded homotopy groups.
\item Next, we proceed by induction on weight. Assuming the claim holds in weights $\le n$ for all finite sets $T$, we reduce the claim in weight $n+1$ for all finite sets $T$ to the case $T=\{t\}$.
\item Finally, we treat the case $T=\{t\}$ using transfer arguments and an analysis of additive operations in $\delta$-rings.
\end{itemize}

We begin with the reduction to the case $T=\{t\}$.
\begin{proposition}
\label{reduce}
    Suppose that
    \[\Theta_{\le n}\colon \Free^{\delta}_{R_*}(t)^{\wt\le n}\too\pi_*(\Free^{\mathbb E_\infty}_{R}(t)^{\wt\le n})\]
    is an isomorphism. Then 
     \[\Theta_{\le n}\colon \Free^{\delta}_{R_*}(T)^{\wt\le n}\too \pi_*(\Free^{\mathbb E_\infty}_{R}(T)^{\wt\le n})\]
     is an isomorphism for all finite sets $T$.
\end{proposition}
\begin{proof}
     Suppose $|T|=2$, as the general case is the same. We will show $\Theta$ is an isomorphism on weight $n$. On the one hand, in light of \Cref{Deltacoproducts}, distributing the tensor product over the weight grading gives
     \[\Free^{\delta}_{R_*}(x,y)^{\wt=n}\cong\bigoplus_{i+j=n}\Free^{\delta}_{R_*}(x)^{\wt=i}\otimes_{R_*} \Free^{\delta}_{R_*}(y)^{\wt=j}.\]
     On the other hand, by the formula for the free commutative algebra on two elements we have
     \[\Free^{\mathbb E_\infty}_{R}(x,y)^{\wt=n}\simeq\bigoplus_{i+j=n}\Free^{\mathbb E_\infty}_{R}(x)^{\wt=i}\otimes \Free^{\mathbb E_\infty}_{R}(y)^{\wt=j}.\]
     Since the homotopy groups of $\Free^{\mathbb E_\infty}_{R}(x)^{\wt=i}$ and $\Free^{\mathbb E_\infty}_{R}(y)^{\wt=j}$ are free $R_*$-modules by assumption, they are free $R$-modules. There is thus a K\"unneth isomorphism
     \[\bigoplus_{i+j=n}\pi_*(\Free^{\mathbb E_\infty}_{R}(x)^{\wt=i})\otimes_{R_*}\pi_*(\Free^{\mathbb E_\infty}_{R}(y)^{\wt=j})\iso \pi_*(\Free^{\mathbb E_\infty}_{R}(x,y)^{\wt=n}).\]
     $\Theta_n$ is the composition
     \begin{align*}
      \Free^{\delta}_{R_*}(x,y)^{\wt=n}&\iso\bigoplus_{i+j=n}\Free^{\delta}_{R_*}(x)^{\wt=i}\otimes_{R_*} \Free^{\delta}_{R_*}(y)^{\wt=j}\\
      &\too\bigoplus_{i+j=n}\pi_*(\Free^{\mathbb E_\infty}_{R}(x)^{\wt=i})\otimes_{R_*}\pi_*(\Free^{\mathbb E_\infty}_{R}(y)^{\wt=j})\\
      &\iso \pi_*(\Free^{\mathbb E_\infty}_{R}(x,y)^{\wt=n}).
      \end{align*}
      Since the middle map is a direct sum of isomorphisms by assumption, the claim follows.
\end{proof}
At this point, we have reduced the proof of \Cref{theoremBproof} to the following proposition:
\begin{proposition}
    Suppose that
    \[\Theta_{< n}\colon \Free^{\delta}_{R_*}(T)^{\wt<n}\too\pi_*(\Free^{\mathbb E_\infty}_{R}(T)^{\wt<n})\]
    is an isomorphism for all finite sets $T$. Then 
     \[\Theta_n\colon \Free^{\delta}_{R_*}(t)^{\wt=n}\too \pi_*(\Free^{\mathbb E_\infty}_{R}(t)^{\wt=n})\]
     is an isomorphism.
\end{proposition}
We break up the proof into surjectivity (\Cref{surjective}) and injectivity (\Cref{injective}).

\begin{lemma}
\label{surjective}
    The map
    \[\Theta_n\colon \Free^{\delta}_{R_*}(t)^{\wt=n}\too\pi_*(\Free^{\mathbb E_\infty}_{R}(t)^{\wt=n})\]
    is surjective.
\end{lemma}

\begin{proof}We follow \cite{hoo}. First suppose $n=pk$. Treating `$t^p$',`$\delta(t)$' as formal symbols, there is a canonical homomorphism \[\Free^{\delta}_{R_*}(t^p,\delta(t))\too \Free^{\delta}_{R_*}(t)\]
    determined by sending $t^p\mapsto t^p$ and $\delta(t)\mapsto \delta(t)$. Note that it does \textit{not} respect the weight grading on both sides, as it sends weight $k$ to weight $pk$.
 There is a commutative diagram
\[\begin{tikzcd}
	{\Free^{\delta}_{R_*}(t^p,\delta(t))^{\wt=k}} & {\pi_*\Free^{\mathbb E_\infty}_{R}(t^p,\delta(t))^{\wt=k}} \\
	{\Free^{\delta}_{R_*}(t)^{\wt=pk}} & {\pi_*\Free^{\mathbb E_\infty}_{R}(t)^{\wt=pk}}
	\arrow["{\Theta_k}", from=1-1, to=1-2]
	\arrow[from=1-1, to=2-1]
	\arrow[from=1-2, to=2-2]
	\arrow["{\Theta_{pk}}", from=2-1, to=2-2]
\end{tikzcd}\]
Since $\Theta_k$ is an isomorphism by induction, to show that $\Theta_{pk}$ is surjective it suffices to show that the right vertical map is surjective. By dualizing \Cref{basis}, we see that $t^p,\theta(t)$ are a basis of $R[B\Sigma_p]$. The right vertical map identifies with the composition
\[\Free^{\mathbb E_\infty}_{R}(t^p,\delta(t))^{\wt=k}\iso  \mathrm{Free}^{\mathbb E_\infty}_R(R[B\Sigma_p])^{\wt=k}= R[B\Sigma_p]^{\otimes k}_{h\Sigma_k}\simeq  R[B\Sigma_p\wr\Sigma_k]\longrightarrow R[B\Sigma_{pk}].\]
Since $\Sigma_p\wr\Sigma_k\subset\Sigma_{pk}$ has index prime to $p$ and $\widehat\tsadi_1$ is $p$-complete, the final map is surjective on $\pi_*$. Therefore the composition is surjective, as desired.

The case in which $n$ is prime to $p$ is similar, except one now begins by considering the homomorphism 
\[\Free^{\delta}_{R_*}(x,y)\too \Free^{\delta}_{R_*}(t)\]
determined by taking $x\mapsto t$ and $y\mapsto t$.
\end{proof}
We now turn to injectivity. 
\begin{lemma}
\label{injective}
   The map
    \[\Theta_n\colon \Free^{\delta}_{R_*}(t)^{\wt=n}\too\pi_*(\Free^{\mathbb E_\infty}_{R}(t)^{\wt=n})\]
    is injective.
\end{lemma}
\begin{proof}
    We leverage the coaddition $\Delta$ on $\Free^{\delta}_{R_*}(t)$ to reduce the claim to injectivity on lower weights. There is a commutative diagram
\[\begin{tikzcd}[ampersand replacement=\&,column sep=huge]
	{\Free^{\delta}_{R_*}(t)} \& {\pi_*(\Free^{\mathbb E_\infty}_{R}(t))} \\
	{\Free^{\delta}_{R_*}(x,y)} \& {\pi_*(\Free^{\mathbb E_\infty}_{R}(x,y))}
	\arrow["\Theta", from=1-1, to=1-2]
	\arrow["\Delta"', from=1-1, to=2-1]
	\arrow["\Delta", from=1-2, to=2-2]
	\arrow["\Theta", from=2-1, to=2-2]
\end{tikzcd}\]
in which all maps respect the weight grading. Suppose $\omega\in \Free^{\delta}_{R_*}(t)^{\wt=n}$ is in the kernel of $\Theta_n$. Recall that
\[\Free^{\delta}_{R_*}(x,y)^{\wt=n}\cong\bigoplus_{i+j=n}\Free^{\delta}_{R_*}(x)^{\wt=i}\otimes_{R_*} \Free^{\delta}_{R_*}(y)^{\wt=j}.\]
Since $\Delta(\omega)$ lies in $\Free^{\delta}_{R_*}(x,y)^{\wt=n}$, we can write
\[\Delta(\omega)=\omega\otimes 1+1\otimes\omega+\widetilde\Delta(\omega)\]
where $\widetilde\Delta(\omega)$ lies in the summand of mixed weights. Since $\Theta_n(\omega)=0$, we have
\begin{align*}0&=\Delta(\Theta_n(\omega)\\
&=\Theta_n(\Delta(\omega))\\
&={\Theta_n(\omega)\otimes1}+{1\otimes \Theta_n(\omega)}+\Theta_n(\widetilde\Delta(\omega))\\
&=0+0+\Theta_n(\widetilde\Delta(\omega)).
\end{align*}
By induction, $\Theta_n$ is injective on the submodule of $\Free^{\delta}_{R_*}(x,y)^{\wt=n}$ of mixed weights, so $\widetilde\Delta(\omega)=0$. We learn that $\omega$ is a primitive. But the only primitives in $\Free^{\delta}_{R_*}(t)$ are iterates of the Frobenius 
\[\omega=c\cdot \psi^{ r}(t),\]
for $c\in R_*$ (\Cref{primitive_classification}). Since $\Theta_n(\omega)=0$, the power operation $c\psi^{\circ r}$ acts by zero on classes of degree $0$ in all commutative $R$-algebras in $\cC$. But evaluating it on $1\in\pi_0R$ gives
\[c=c\cdot\psi^r(1)=0,\]
so $\omega=0$.
\end{proof}

\subsection{$\lambda$-Rings and Redshift}
\label{lambda_redshift}
In this section we relate operations on the $K$-theory of rational symmetric monoidal categories $\cC$ to the semiadditive power operation $\theta$ studied above. We begin with the well-known fact that $K_0(\cC)$ is a $\lambda$-ring. In particular, we recall the construction of the $\lambda$-ring operations on the level of the $K$-theory space of $\cC$ (this is one of the motivating examples discussed in \cite{polynomial}).
\begin{lemma}
    Let $\cC\in \Cat_\perf$ be $\QQ$-linear. Then $\cC$ has $1$-finite colimits.
\end{lemma}
\begin{proof}
We will show that $\cC$ has $BG$-shaped colimits by passing to the ind-completion $\Ind(\cC)$. Since $|G|$ is invertible and $\Ind(\cC)$ has $1$-finite colimits, $BG$ is $\Ind(\cC)$-acyclic and $\Ind(\cC)$-ambidextrous. Then by \cite[Proposition 4.3.4]{ambiheight}, there is an idempotent endomorphism 
\[\varepsilon\colon \id_{\Ind(\cC)^{BG}}\too\id_{\Ind(\cC)^{BG}}\] 
such that splitting $\epsilon$ realizes the projection onto the essential image of the fully-faithful functor \[BG^*:\Ind(\cC)\to\Ind(\cC)^{BG}.\] Additionally, this projection identifies with the left adjoint of $BG^*$ (\cite[Proposition 4.3.1]{ambiheight}). 

We wish to verify that this projection restricts to give a left adjoint of $BG^*\colon \cC\to\cC^{BG}$. Given $X$ in $\cC^{BG}\subset\Ind(\cC)^{BG}$, splitting $\varepsilon$ on $X$ commutes with passing to the underlying object. Since $\cC$ is idempotent-complete, the result already lies in $\cC$, as desired.
\end{proof}
\begin{proposition}
\label{rational_lambda}
    Let $\cC\in\CAlg(\Cat_\perf)$ be $\QQ$-linear. For finite-dimensional $\Sigma_n$-representations $V$, the operations
    \[X\longmapsto (V\otimes X^{\otimes n})_{h\Sigma_n}\]
    are well-defined on $\uK(\cC)$ and endow $   K_0(\cC)$ with the structure of a $\lambda$-ring. 
\end{proposition}

\begin{proof}First note that every $\Sigma_n$-representation $V$ is defined over $\QQ$, so that $V$ gives an object in $\cC^{B\Sigma_n}$. Since $\cC$ is closed under $1$-finite colimits, there are operations $\widehat \beta_G^V$ defined on $\uK(\cC)$ (\Cref{poly_operations_k_theory}). In particular, we can take $\lambda^n=\widehat\beta_{\Sigma_n}^\text{sgn}$ (or $\sigma^n=\widehat\beta_{\Sigma_n}^\text{triv}$). The fact that these satisfy the required $\lambda$-ring identities on $K_0(\cC)$ is well known (for example, as in \cite[Section 4.2]{heinloth2005notefunctionalequationszeta} or \cite{comparisonexteriorpoweroperations}).
\end{proof}
\begin{rem}
Suppose that $R$ is a discrete commutative ring (not necessarily a $\QQ$-algebra). We note that $K_0(R)$ is still a $\lambda$-ring (for example, see \cite{comparisonexteriorpoweroperations} for various constructions). However, we do not know how to make the comparison of \Cref{redshift_delta} in this case.
\end{rem}
We now relate this $\lambda$-ring structure on $K(\cC)$ to the semiadditive power operation $\theta$. Since the $\lambda$-ring operations on $K_0(\cC)$ are defined on the underlying space $\uK(\cC)$, there are operations
\[\delta_p\colon \uK(\cC)\too\uK(\cC)\]
that give the $p$-derivations of \Cref{defn3} on $K_0(\cC)$. Working at a fixed prime $p$, this single operation $\delta_p$ is exactly the `$p$-typical part' of the $\lambda$-ring structure on $K(\cC)$ (\Cref{ptypical}).
\begin{cor}
\label{redshift_delta}
    Let $\cC\in\CAlg(\Cat_\perf)$ be $\QQ$-linear. Then there is a commutative diagram
\[\begin{tikzcd}[ampersand replacement=\&]
	{\uK(\cC)} \& {\uK(\cC)} \\
	{\uK_{T(1)}(\cC)} \& {\uK_{T(1)}(\cC)}
	\arrow["\delta_p", from=1-1, to=1-2]
	\arrow[from=1-1, to=2-1]
	\arrow[from=1-2, to=2-2]
	\arrow["\theta", from=2-1, to=2-2]
\end{tikzcd}\]
where $\delta_p$ is the $p$-derivation associated to the $\lambda$-ring structure on $K(\cC)$ and $\theta$ is the operation of \Cref{theta}.
\end{cor}

\begin{proof}
    By \Cref{delta_rep} and the construction of the $\lambda$-ring operations on $\uK(\cC)$ in \Cref{rational_lambda}, we see that the operation $\delta_p$ on $\uK(\cC)$ is given by
    \[\delta_p(x)=\frac{1}{p-1}\big(\widehat\beta_{C_p}^{[p]}(x)-x^p\big).\]
    Next, note that the cardinality $|BC_p|$ acts on $\cC$ by $BC_p\otimes-$ (\cite[Example 2.2.8]{ambiheight}), which is the identity because $\cC$ is rational. Since $K^{[1]}$ is a $1$-semiadditive functor, we get that $|BC_p|$ is the identity on $K^{[1]}(\cC)$. Thus the formula for $\theta$ (\Cref{theta}) simplifies to the above expression, and the result follows from \Cref{ktheory_comparison} and \Cref{lnf_comparison}.
\end{proof}

By \Cref{theoremBproof}, power operations on $T(1)$-local commutative ring spectra are precisely $\delta$-ring operations. Since $\delta$-rings are $p$-typical $\lambda$-rings (\Cref{ptypical}), power operations in chromatic height $1$ are exactly the $p$-typical version of power operations on the $K$-theory of rational commutative ring spectra. This is an instance of the redshift philosophy. \Cref{redshift_delta} further shows that the localization map $K(R)\to K_{T(1)}(R)$ indeed intertwines these $p$-typical $\lambda$-ring structures.

\bibliographystyle{alpha}
\bibliography{bibliography}
\end{document}